\documentclass[11pt]{amsart}
\usepackage{amsmath,amssymb,latexsym,cite}
\usepackage[small]{caption}
\usepackage{graphicx,color,mathrsfs,tikz}
\usepackage{subfigure,color}
\usepackage{cite}
\usepackage[colorlinks=true,urlcolor=blue,
citecolor=red,linkcolor=blue,linktocpage,pdfpagelabels,
bookmarksnumbered,bookmarksopen]{hyperref}
\usepackage[italian,english]{babel}
\usepackage{units}
\usepackage{enumitem}
\usepackage[left=2.6cm,right=2.6cm,top=2.8cm,bottom=2.8cm]{geometry}
\usepackage[hyperpageref]{backref}

\usepackage[colorinlistoftodos,prependcaption]{todonotes}

\numberwithin{equation}{section}
\newtheorem{theorem}{Theorem}[section]
\newtheorem{proposition}[theorem]{Proposition}
\newtheorem{lemma}[theorem]{Lemma}
\newtheorem{remark}[theorem]{Remark}

\newtheorem{definition}[theorem]{Definition}
\theoremstyle{definition}
\newtheorem*{theoBM}{Theorem}

\renewcommand{\epsilon}{\eps}
\renewcommand{\i}{{\rm i}}

\newcommand{\N}{{\mathbb N}}
\newcommand{\R}{{\mathbb R}}

\newcommand{\dvg}{{\rm div}}

\newcommand{\eps}{\varepsilon}

\newcommand{\iu}{{\rm i}}

\newcommand{\pnorm}[2][]{\if #1'' \left|#2\right|_p \else \left|#2\right|_{#1} \fi}

\newcommand{\scal}[2]{{#1} \cdot {#2}\,}

\newcommand{\loc}{{\rm loc}}

\newcommand{\C}{\mathbb{C}}

\renewcommand{\theta}{\vartheta}

\title[Magnetic BV functions and the Bourgain-Brezis-Mironescu formula]{Magnetic
	BV functions and the \\ Bourgain-Brezis-Mironescu formula}

\author[A.\ Pinamonti]{Andrea Pinamonti}
\author[M.\ Squassina]{Marco Squassina}
\author[E.\ Vecchi]{Eugenio Vecchi}

\address[A.\ Pinamonti]{Dipartimento di Matematica \newline\indent
	Universit\`a di Trento,
	Via Sommarive 14, 38050 Povo (Trento), Italy}
\email{andrea.pinamonti@unitn.it}

\address[M.\ Squassina]{Dipartimento di Matematica e Fisica \newline\indent
	Universit\`a Cattolica del Sacro Cuore,
	Via dei Musei 41, I-25121 Brescia, Italy}
\email{marco.squassina@dmf.unicatt.it}

\address[E.\ Vecchi]{Dipartimento di Matematica \newline\indent
	Universit\`a di Bologna,
	Piazza di Porta S. Donato 5, 40126, Bologna, Italy}
\email{eugenio.vecchi2@unibo.it}

\thanks{The authors are members of {\em Gruppo Nazionale per l'Analisi Ma\-te\-ma\-ti\-ca, la Probabilit\`a e le loro Applicazioni} (GNAMPA) of the {\em Istituto Nazionale di Alta Matematica} (INdAM). E.V. receives funding from the People Programme (Marie Curie Actions) of the European Union's Seventh
Framework Programme FP7/2007-2013/ under REA grant agreement No.\ 607643 (ERC Grant MaNET `Metric Analysis for Emergent Technologies')}



\begin{document}
\hyphenation{Spia-na-to}
\begin{abstract}
	We prove a general magnetic Bourgain-Brezis-Mironescu formula which 
	extends the one 
	obtained in \cite{BM} in the Hilbert case setting.\ In particular, after developing a rather complete theory of magnetic bounded variation functions, we prove the 
	validity of the formula in this class.
\end{abstract}

\maketitle

\begin{center}
	\begin{minipage}{7.5cm}
		\small
		\tableofcontents
	\end{minipage}
\end{center}

\medskip


\section{Introduction}
The celebrated Bourgain-Brezis-Mironescu formula, $(BBM)$ in short, appeared 
for the first time in \cite{bourg,bourg2}, and provided a new characterization for
functions in the Sobolev space $W^{1,p}(\Omega)$,  with $p\geq 1$ and for
$\Omega\subset\R^N$ being a smooth bounded domain. To this aim, the authors of \cite{bourg,bourg2} perform
a careful study of the limit properties of the Gagliardo semi-norm
defined for the fractional Sobolev spaces $W^{s,p}(\Omega)$ with $0<s<1$.
In particular, they considered the limit as $s\nearrow 1$.
To be more precise, for any $W^{1,p}(\Omega)$ it holds  
\begin{equation}
\tag{$BBM$}
\lim_{s\nearrow 1}(1-s)\int_{\Omega}\int_{\Omega}\frac{|u(x)-u(y)|^p}{|x-y|^{N+ps}}dxdy=
Q_{p,N}\int_{\Omega}|\nabla u|^pdx,
\end{equation}
where $Q_{p,N}$ is defined by
\begin{equation}
\label{valoreK}
Q_{p,N}=\frac{1}{p}\int_{{\mathbb S}^{N-1}}|{\boldsymbol \omega}\cdot h|^{p}d\mathcal{H}^{N-1}(h),
\end{equation}
where ${\mathbb S}^{N-1} \subset \R^N$ denotes the unit sphere 
and ${\boldsymbol \omega}$ is an arbitrary unit vector of $\R^N$. 
This also allows to get the stability of (variational) eigenvalues for the fractional $p$-Laplacian operator as $s\nearrow 1$, see
\cite{brapasq}.
We recall that characterizations similar to $(BBM)$
when $s\searrow 0$ were obtained in \cite{mazia, mazia2}.\\
In the following years, a huge effort 
in trying to extend the results 
proved in \cite{bourg} has been made.
One of the first extension was achieved by Nguyen in \cite{nguyen06}, where 
he provided a new characterization for
functions in $W^{1,p}(\R^N)$. 
As we already mentioned, the $(BBM)$-formula proved in \cite{bourg}
covered the case of $\Omega\subset \R^N$ being a smooth and bounded
domain, therefore it was quite natural to try to relax the 
assumptions on the open set $\Omega \subset \R^N$: this kind of problem
was recently addressed in \cite{Spe} and \cite{Spe2}, where Leoni and Spector
were able to provide a generalization of the $(BBM)$-formula to \textit{any}
open set $\Omega \subset \R^N$ .
The interest resulted from \cite{bourg} led also to related
new characterizations of Sobolev spaces in non-Euclidean contexts
like the Heisenberg group (see \cite{Barb,Cui}).\\
\indent One of the most challenging problems left open in \cite{bourg}
was to provide similar characterizations for functions of bounded variation.
A positive answer to this question has been given by Davila in \cite{davila} 
and by Ponce in \cite{Ponce}. They  
completed the picture by showing that,
$$
\lim_{s\nearrow 1}(1-s)\int_{\Omega}\int_{\Omega}\frac{|u(x)-u(y)|}{|x-y|^{N+s}}dxdy=
Q_{1,N}|Du|(\Omega),
$$
for every bounded Lipschitz set $\Omega \subset \R^N$ and every $u\in BV(\Omega)$. We also recall
that the extension to any open set proved in \cite{Spe, Spe2}
concerns $BV$ functions as well, see also \cite{PonceSpector}.\\
In order to try to give a more complete overview of the subject,
we have to mention that, parallel to the fractional theory of
Sobolev spaces, there exists a quite developed theory of fractional $s$-perimeters (e.g. \cite{CRS}),
and also in this framework there have been several contributions concerning their
analysis in the limits 
$s\nearrow 1$ and $s\searrow 0$ (see e.g. \cite{CV,AmbDepMart, Ludwig1, Ludwig2, Dip, FerPin}).\\
Very recently the results we have mentioned have been discovered to
have interesting applications in image processing, 
see for instance \cite{BHN,BHN2,BHN3,bre-linc}.
One of the latest generalizations of $(BBM)$ appeared
very recently in \cite{BM} in the context of magnetic Sobolev spaces $W_{A}^{1,2}(\Omega)$.
In fact, an important role in the study of particles which interact 
with a magnetic field $B=\nabla\times A$, $A:\R^3\to\R^3$,  is assumed by another {\em extension} of the Laplacian, namely the {\em magnetic Laplacian} $(\nabla-\iu A)^2$ (see \cite{AHS,reed,LL}), yielding to nonlinear Schr\"odinger equations like
\begin{equation}
\label{mageq}
- (\nabla-\iu A)^2 u + u = f(u),
\end{equation}
which have been extensively studied (see e.g.\ \cite{arioliSz} and references therein), where $(\nabla-\iu A)^2$
is defined in weak sense as the differential of the integral functional
\begin{equation}
\label{Anorma}
W_{A}^{1,2}(\Omega) \ni u\mapsto \int_{\Omega}|\nabla u-\i A(x)u|^2dx.
\end{equation}
If $A:\R^N\to\R^N$ is a smooth function and $s \in (0,1)$,
a non-local magnetic counterpart of \eqref{mageq}, i.e.
\begin{equation*}
(-\Delta)^s_Au(x)=c(N,s) \lim_{\eps\searrow 0}\int_{B^c_\eps(x)}\frac{u(x)-e^{\i (x-y)\cdot A\left(\frac{x+y}{2}\right)}u(y)}{|x-y|^{N+2s}}dy,
\qquad \lim_{s\nearrow 1}\frac{c(N,s)}{1-s}=\frac{4N\Gamma(N/2)}{2\pi^{N/2}},
\end{equation*}
was introduced  in \cite{piemar,I10} for complex-valued functions. We point out that
$(-\Delta)^s_A$  coincides with the usual 
fractional Laplacian for $A=0$.
The motivations for the introduction of this operator are carefully 
described in \cite{piemar,I10} and fall into
the framework of the general theory of L\'evy processes. 
It is thus natural wondering about the consistency of 
the norms associated with the above fractional magnetic operator
in the singular limit $s\nearrow 1$,
with the energy functional \eqref{Anorma}. We point out that the case $s\nearrow 0$ has been studied in \cite{Lim0}.

The aim of this paper is to continue the study of the validity of a magnetic
counterpart of $(BBM)$, extending the results of \cite{BM} to arbitrary magnetic
fractional Sobolev spaces and to magnetic $BV$ functions. We refer the reader to  
Sections~\ref{magnSec} and \ref{BVsec} for the definitions.
On the other hand, while for $p\geq 1$ the spaces $W^{1,p}_A(\Omega)$ have
a wide background, to the best of our knowledge no notion
of {\em magnetic bounded variations space} 
containing $W^{1,1}_A(\Omega)$ seems to be 
previously available in the literature.
\vskip3pt
\noindent
As already recalled, this indeed holds for the Hilbert case $p=2$, 
as stated in the following 

\begin{theoBM}[M.\ Squassina, B.\ Volzone \cite{BM}]
	Let $\Omega\subset\R^N$ be an open and bounded set with Lipschitz boundary and let $A\in C^2(\bar{\Omega}, \R^N)$.\ Then,
	for every  $u\in W^{1,2}_{A}(\Omega)$, we have
	$$
	\lim_{s\nearrow 1}(1-s)\int_{\Omega}\int_{\Omega}\frac{|u(x)-e^{\i (x-y)\cdot A\left(\frac{x+y}{2}\right)}u(y)|^2}{|x-y|^{N+2s}}dxdy=
	Q_{2,N}\int_{\Omega}|\nabla u-\i A(x)u|^2dx,
	$$
	where $Q_{2,N}$ is the positive constant defined in \eqref{valoreK} with $p=2$.
\end{theoBM}

\indent The goal of this paper is twofold: 
first we aim to extend this formula to the case of general magnetic spaces 
$W_{A}^{1,p}$ for $p \geq 1$, and secondly we introduce a suitable
notion of {\em magnetic bounded variation} $|Du|_A(\Omega)$ and we prove that a $(BBM)-$ formula holds also in that case. 

\vskip4pt
\noindent
In order to state the main result we need to introduce some notation:
let $p\geq 1$ be fixed and let us consider the normed space 
$(\C^N, |\cdot|_{p})$,  with
\begin{equation}\label{p-norm}
 |z|_p:=\left(|(\Re z_1,\ldots, \Re z_N)|^p+|(\Im z_1,\ldots, \Im z_N)|^p\right)^{1/p},
 \end{equation}
 where $|\cdot|$ is the Euclidean norm of $\R^N$ and
 $\Re a$,$\Im a$ denote the real and imaginary parts of $a\in\mathbb{C}$
 respectively. Notice that $|z|_p=|z|$ whenever $z 	\in \R^N$, which makes our next 
statements consistent with the case $A=0$ and $u$ being 
a real valued function \cite{bourg,bre,davila,Ponce}. 
 


\begin{theorem}[General magnetic Bourgain-Brezis-Mironescu limit]
	\label{main}
Let $A:\R^N\to \R^N$ be of class $C^2$.
Then, for any bounded extension domain $\Omega\subset\R^N$
$$
\lim_{s\nearrow 1}(1-s)\int_{\Omega}\int_{\Omega}\frac{|u(x)-e^{\i (x-y)\cdot A\left(\frac{x+y}{2}\right)}u(y)|_1}{|x-y|^{N+s}}dxdy=
Q_{1,N}|Du|_A(\Omega),
$$
for all $u\in BV_A(\Omega)$, where $Q_{p,N}$ is defined in \eqref{valoreK}. Furthermore, for any $p\geq 1$ and any Lipschitz bounded domain $\Omega\subset\R^N$
$$
\lim_{s\nearrow 1}(1-s)\int_{\Omega}\int_{\Omega}\frac{|u(x)-e^{\i (x-y)\cdot A\left(\frac{x+y}{2}\right)}u(y)|^p_p}{|x-y|^{N+ps}}dxdy=
Q_{p,N}\int_{\Omega}|\nabla u-\i A(x)u|^p_p\, dx,
$$
for all $u\in W^{1,p}_A(\Omega)$.
\end{theorem}
We refer to Definition \ref{exdomain} for a precise explanation of \textit{extension domain}. 
We stress that the definitions of both the magnetic Sobolev spaces
$W_{A}^{1,p}(\Omega)$ and of the magnetic $BV$ spaces $BV_{A}(\Omega)$ 
made in Sections~\ref{magnSec} and ~\ref{BVsec} 
are consistent, in the case of zero magnetic potential $A$, with the classical spaces $W^{1,p}(\Omega)$ 
and $BV(\Omega)$, respectively. Moreover,  it holds $|Du|_A(\Omega)=|Du|(\Omega)$, 
so that Theorem~\ref{main} is consistent with the classical formulas of \cite{bourg,davila,Ponce}.\\
In particular, in the spirit of \cite{bre}, as a byproduct of Theorem~\ref{main}, 
if $\Omega\subset\R^N$ is a smooth bounded domain, 
$A:\R^N\to \R^N$ is of class $C^2$ and we have
$$
\lim_{s\nearrow 1}(1-s)\int_{\Omega}\int_{\Omega}\frac{|u(x)-e^{\i (x-y)\cdot A\left(\frac{x+y}{2}\right)}u(y)|^p_p}{|x-y|^{N+ps}}dxdy=0,\quad\,\, 
u\in W^{1,p}_A(\Omega),
$$
then we get
\begin{equation*}
\begin{cases}
\nabla \Re u=-\Im u A, &\\
\nabla \Im u=\Re u A,
\end{cases}
\end{equation*}
namely the direction of $\nabla \Re u,\nabla \Im u$ is that of the 
magnetic potential $A$.\ In the particular case $A=0$, consistently with the results of \cite{bre},  this implies that $u$ is a constant function. \\
We finally notice that for a Borel set $E\subset\Omega$,  denoting $E^c=\Omega\setminus E$, the quantity
\begin{align*}
P_s(E;A):&=
\frac{1}{2}\int_{E}\int_{E}\frac{|1-e^{\i (x-y)\cdot A\left(\frac{x+y}{2}\right)}|_1}{|x-y|^{N+s}}dxdy+
\frac{1}{2}\int_{E}\int_{E^c}\frac{1}{|x-y|^{N+s}}dxdy \\
&+\frac{1}{2}\int_{E^c}\int_{E}\frac{|e^{\i (x-y)\cdot A\left(\frac{x+y}{2}\right)}|_1}{|x-y|^{N+s}}dxdy,
\end{align*}
plays the r\v ole of a nonlocal $s$-perimeter of $E$ depending on $A$, which
reduces for $A=0$ to the classical notion of fractional $s$-perimeter of $E$ in $\Omega$
$$
P_s(E)=\int_{E}\int_{E^c}\frac{1}{|x-y|^{N+s}}dxdy.
$$
Then, the main result Theorem~\ref{main} reads as 
$$
\lim_{s\nearrow 1}(1-s)P_s(E,A)=Q_{1,N}|D {\bf 1}_E|_A(\Omega).
$$
\vskip3pt

\indent The structure of the paper is as follows. In Section~\ref{magnSec}
we introduce magnetic Sobolev spaces $W^{1,p}_A(\Omega)$. In Section~\ref{BVsec}
we define the magnetic $BV$ space $BV_{A}(\Omega)$ and we prove that several classical results for BV functions hold also for functions belonging to $BV_{A}(\Omega)$. In particular, we prove a structure result (Lemma~\ref{Struttura}),
a result about the extension to $\R^N$ for Lipschitz domains (Lemma~\ref{ExtDom-new}), 
the semi-continuity of the variation (Lemma~\ref{semic}),
a magnetic counterpart of the classical Anzellotti-Giaquinta approximation Theorem
(Lemma~\ref{Approx}) and, finally, a compactness result (Lemma~\ref{compactness}).
In Sections \ref{main_res}, \ref{main_res2} and \ref{sec6} we 
finally prove Theorem~\ref{main}.

\section{Magnetic Sobolev spaces}
\label{magnSec}
\noindent
In order to avoid confusion with the different uses of
the symbol $v \cdot w$, we define
\begin{equation*}
v \cdot w := 
\sum_{i=1}^{N} (\Re v_i + \i \Im v_i)(\Re w_i + \i \Im w_i), \qquad \textrm{if $v,w\in\C^N$.}
\end{equation*}						
Let $\Omega$ be an open set of $\R^N$. For any $p\geq 1$ we denote by $L^p(\Omega,\C)$ 
the Lebesgue space of complex valued functions $u:\Omega\to\C$ such that
$$
\|u\|_{L^p(\Omega)}=\left(\int_{\Omega}|u(x)|_p^pdx\right)^{1/p}<\infty,
$$
where $|\cdot|_p$ is as in \eqref{p-norm}.
For a locally bounded function $A:\R^N\to\R^N$, we consider the semi-norm 
$$
[u]_{W^{1,p}_A(\Omega)}:=\Big(\int_{\Omega}|\nabla u-\i A(x)u|^p_pdx\Big)^{1/p},
$$
and define $W^{1,p}_A(\Omega)$ as the space of functions $u\in L^p(\Omega,\C)$ such that  $[u]_{W^{1,p}_A(\Omega)}<\infty$ with norm
$$
\|u\|_{W^{1,p}_A(\Omega)}:=\Big(\|u\|_{L^p(\Omega)}^p+[u]_{W^{1,p}_A
	(\Omega)}^p\Big)^{1/p}.
$$
The space $W^{1,p}_{0,A}(\Omega)$ will denote the closure of the space $C^\infty_c(\Omega)$ in $W^{1,p}_A(\Omega)$.
For any $s\in (0,1)$ and $p\geq 1$, the magnetic Gagliardo semi-norm is defined as 
$$
[u]_{W^{s,p}_A(\Omega)}:=\Big(\int_{\Omega}\int_{\Omega}\frac{|u(x)-e^{\i (x-y)\cdot A\left(\frac{x+y}{2}\right)}u(y)|^p_p}{|x-y|^{N+ps}}dxdy\Big)^{1/p}.
$$
We denote by $W^{s,p}_A(\Omega)$ the space of functions $u\in L^p(\Omega,\C)$ such that  $[u]_{W^{s,p}_A(\Omega)}<\infty$ normed with 
$$
\|u\|_{W^{s,p}_A(\Omega)}:=\left(\|u\|_{L^p(\Omega)}^p+[u]_{W^{s,p}_A(\Omega)}^p\right)^{1/p}.
$$
For $A=0$ this is consistent with the usual space $W^{s,p}(\Omega)$ with norm $\|\cdot\|_{W^{s,p}(\Omega)}$.

\section{Magnetic BV spaces}
\label{BVsec}
\noindent
In this section we introduce a suitable notion of magnetic bounded variation functions.
Let $\Omega$ be an open set of $\R^N$.
We recall that a real-valued function $u\in L^1(\Omega)$ 
is of bounded variation, and we shall write $u\in BV(\Omega)$, if
$$
|Du|(\Omega)=\sup\left\{\int_{\Omega}  u(x) \dvg\varphi(x)dx \ |\ \varphi\in C_c^{\infty}(\Omega,\R^N),\ \|\varphi\|_{L^\infty(\Omega)}\leq 1\right\}<\infty.
$$
The space $BV(\Omega)$ is endowed with the norm
$$
\|u\|_{BV(\Omega)} := \|u\|_{L^{1}(\Omega)} + |Du|(\Omega).
$$
The space of complex-valued bounded variation functions 
$BV(\Omega,\mathbb{C})$ 
is defined as the class of Borel functions 
$u:\Omega\to \mathbb{C}$ such that $\Re u, \Im u\in BV(\Omega)$. 
The $\mathbb{C}$-total variation of $u$ is defined by
\[
|Du|(\Omega):=|D\Re u|(\Omega)+|D\Im u|(\Omega).
\]
More generally, it is possible to define a notion of variation for functions $u: \Omega\to E$ where $\Omega\subset\R^N$ is an open set and $(E,d)$ is a locally compact metric space. We refer the interested reader to \cite{Ambrosio}.

We are now ready to define the magnetic $BV$ functions.
\begin{definition}[$A-$bounded variation functions]
Let $\Omega\subset \R^N$ be an open set and 
$A:\R^N\to\R^N$ a locally bounded function. A function $u\in L^1(\Omega,\mathbb{C})$ is said to be of $A$-bounded 
variation and we write $u\in BV_A(\Omega)$, if
\[
|Du|_A(\Omega):=C_{1,A,u}(\Omega)+C_{2,A,u}(\Omega)<\infty,
\]
where we have set 
\begin{align*}
& C_{1,A,u}(\Omega):=\sup\left\{\int_{\Omega}  \Re u(x) \dvg \varphi(x)-
A(x)\cdot\varphi(x)\, \Im u(x) dx \ |\ \varphi\in C_c^{\infty}(\Omega,\R^N),\ \|\varphi\|_{L^\infty(\Omega)}\leq 1\right\}, \\
& C_{2,A,u}(\Omega):=\sup\left\{\int_{\Omega}  \Im u(x) \dvg \varphi(x)+
A(x)\cdot\varphi(x) \, \Re u(x) dx \ |\ \varphi\in C_c^{\infty}(\Omega,\R^N),\ \|\varphi\|_{L^\infty(\Omega)}\leq 1\right\}.
\end{align*}
A function $u\in L^1_{\loc}(\Omega,\mathbb{C})$ is said to be of locally $A$-bounded variation and we write $u\in BV_{A,\loc}(\Omega)$, 
provided that it holds
\[
|Du|_A(U)<\infty, \qquad \textrm{for every open set $U\Subset \Omega$}.
\]
\end{definition}

\noindent
We stress that for $A\equiv 0$, the previous definition 
is consistent with the one of $BV(\Omega)$. 
In order to justify our definition, we will collect in the
following some properties
of the space $BV_A(\Omega)$. These properties are the natural generalization to the magnetic setting of the classical theory \cite{ABF,EG,Z}.

\begin{lemma}[Extension of $|Du|_{A}|$]
	Let $\Omega \subset \R^N$ be an open and bounded set, $A:\R^N\to\R^N$ locally bounded and $u \in BV_{A}(\Omega)$. Let $E\subset \Omega$ be a Borel set then
	\begin{align*}
	|Du|_A(E):=\inf\{C_{1,A,u}(U)\ |\ E\subset U,\ U\subset\Omega\ \textrm{open}\}+\inf\{C_{2,A,u}(U)\ |\ E\subset U,\ U\subset\Omega\ \textrm{open}\}
	\end{align*}
	extends $|Du|_A(\cdot)$ to a Radon measure in $\Omega$. For any open set $U\subset\Omega$, $C_{1,A,u}(U)$ and $C_{2,A,u}(U)$ are defined requiring the test functions to be supported in $U$ and $|Du|_{A}(\emptyset):=0$.
	\begin{proof}
		We note that 
		\[
		\nu_1(E):=\inf\{C_{1,A,u}(U)\ |\ E\subset U,\ U\subset\Omega\ \mbox{open}\}
		\]
		is the variation measure associated with 
		$$
		\varphi\mapsto \int_{\Omega}\Re u(x) \dvg \varphi(x) - A(x)\cdot \varphi(x)\, \Im u(x) \, dx,
		$$ 
		and by \cite[Theorem 1.38]{EG} it is a Radon measure. The same argument applies to 
		\[
		\nu_2(E):=\inf\{C_{2,A,u}(U)\ |\ E\subset U,\ U\subset\Omega\ \mbox{open}\}
		\]
		and the thesis follows.
	\end{proof}
\end{lemma}


\begin{lemma}[Local inclusion of Sobolev functions]
	Let $\Omega\subset\R^N$ be an open set. Let $A:\R^N\to\R^N$ be locally bounded. Then
	\[
	W^{1,1}_{\loc}(\Omega)\subset BV_{A,\loc}(\Omega).
	\]
	\begin{proof}
		Let $u \in W^{1,1}_{\loc}(\Omega)$, $U \Subset \Omega$ open  
		and consider $\varphi \in C_{c}^{\infty}(U, \R^N)$ 
		with $\|\varphi\|_{L^\infty(U)}\leq 1$. Then 
		\begin{align*}
		\int_{U}&\Re u(x) \, \mathrm{div}\,\varphi(x) - A(x)\cdot\varphi(x) \, \Im u(x)\, dx + 
		\int_{U} \Im u(x) \, \mathrm{div}\,\varphi(x) + A(x)\cdot\varphi(x) \, \Re u(x) \,dx \\
		& =-\int_{U}\left( \nabla \Re u(x) + A(x)\, \Im u(x)\right) \cdot\varphi(x) \,dx 
		- \int_{U} \left( \nabla \Im u(x) - A(x)\, \Re u(x)\right)\cdot\varphi(x)\, dx \\
		&\leq \int_{\bar U}|\nabla \Re u(x) + A(x)\, \Im u(x)|\,dx 
		+ \int_{\bar U}|\nabla \Im u(x) - A(x)\, \Re u(x)|\,dx\\
		&\leq \int_{\bar U}|\nabla \Re u(x)|\, dx + \int_{\bar U}|\nabla \Im u(x)|\, dx + \|A\|_{L^\infty(\bar U)} \,\left( \int_{\bar U}(|\Re u(x)|+|\Im u(x)|)\,dx \right) < \infty,
		\end{align*}
		which, taking the supremum over $\varphi$, concludes the proof.
	\end{proof}
\end{lemma}

\noindent
Next we prove that for $W^{1,1}_A(\Omega)$ functions the 
magnetic bounded variation semi-norm $|Du|_{A}(\Omega)$ 
boils down to the usual local magnetic semi-norm.

\begin{lemma}[$BV_A$ norm on $W^{1,1}_A$]\label{NormaSmooth}
	\label{Bv-lemmino}
Let $\Omega\subset\R^N$ be an open set.
	Let $A:\R^N\to\R^N$ be locally bounded. Assume that $u\in W^{1,1}_A(\Omega)$.
	Then $u\in BV_A(\Omega)$ and it holds
	\[
	|Du|_{A}(\Omega)=\int_{\Omega} |\nabla u-\i A(x) u|_1 dx.
	\]
	Furthermore, if $u\in BV_A(\Omega)\cap C^\infty(\Omega)$, then $u\in W^{1,1}_A(\Omega)$.
\end{lemma}
\begin{proof}
	If $u\in W^{1,1}_A(\Omega)$, then we have
	$$
	\nabla \Re u+ A\Im u\in L^1(\Omega),\qquad
	\nabla \Im u- A\Re u\in L^1(\Omega).
	$$
	For every $\varphi\in C_c^{\infty}(\Omega,\R^N)$ with $\|\varphi\|_{L^\infty(\Omega)}\leq 1$, we have
	\begin{align*}
	&\left|\int_{\Omega}  \Re u(x) \dvg \varphi(x)-A(x)\cdot \varphi(x) \, \Im u(x) dx\right|\\
	\nonumber
	& = \left|\int_{\Omega} \nabla\Re u(x)\cdot \varphi(x)
	+ A(x)\cdot\varphi(x)\, \Im u(x) dx\right|\leq \int_{\Omega}|\nabla \Re u+ A\Im u| dx,
	\end{align*}
	as well as
	\begin{align*}
	&\left|\int_{\Omega}  \Im u(x) \dvg \varphi(x)+A(x)\cdot\varphi(x) \, \Re u(x) dx\right|\\
	\nonumber
	& =\left|\int_{\Omega} \nabla\Im u(x)\cdot\varphi(x)
	-A(x)\cdot\varphi(x) \, \Re u(x) dx\right|\leq \int_{\Omega}|\nabla \Im u- A\Re u| dx,
	\end{align*}
	which, taking the supremum over $\varphi$, 
	proves $u\in BV_A(\Omega,\mathbb{C})$ and 
	\begin{align}\label{pp}
	|Du|_{A}(\Omega)\leq  \int_{\Omega}|\nabla u-\i A(x) u|_1dx.
	\end{align}
	Defining now $f,g\in L^\infty(\Omega,\R^N)$ by setting
	$$
	f(x):=
	\begin{cases}
	-\frac{\nabla \Re u(x)+ A(x)\Im u(x)}{|\nabla \Re u(x)+ A(x)\Im u(x)|}, & 
	\text{if $x\in\Omega$ and $\nabla \Re u(x)+ A(x)\Im u(x)\neq 0$,} \\
	0, & \text{otherwise,}
	\end{cases}
	$$
	and 
	$$
	g(x):=
	\begin{cases}
	-\frac{\nabla \Im u(x)- A(x)\Re u(x)}{|\nabla \Im u(x)- A(x)\Re u(x)|}, & 
	\text{if $x\in\Omega$ and $\nabla \Im u(x)- A(x)\Re u(x)\neq 0$,} \\
	0, & \text{otherwise,}
	\end{cases}
	$$
	we have that $\|f\|_\infty,\|g\|_\infty\leq 1$.\ By a standard
	approximation result, there exist two sequences 
	$\{\varphi_n\}_{n\in\N},\{\psi_n\}_{n\in \N}\subset C^\infty_c(\Omega,\R^N)$ such that 
	$\varphi_n\to f$ and $\psi_n\to g$ pointwise as $n\to\infty$, 
	with $\|\varphi_n\|_{L^\infty(\Omega)},\|\psi_n\|_{L^\infty(\Omega)}\leq 1$ for all $n \in \N$.\
	By definition of $C_{1,A,u}(\Omega)$, after integration by parts, it follows that, for every $n\geq 1$,
	$$
	C_{1,A,u}(\Omega)\geq -\sum_{i=1}^N\int_{\Omega} \big(\partial_{x_i}\Re u(x)
	+ A^{(i)}(x)\Im u(x) \big)\varphi^{(i)}_n(x)dx.
	$$
	By the Dominated Convergence Theorem and the definition of $f$, 
	letting $n\to \infty$ we obtain
	$$
	C_{1,A,u}(\Omega)\geq \int_{\Omega} |\nabla\Re u(x)
	+ A(x)\Im u(x)|dx.
	$$
	Similarly, using the sequence $\{\psi_n\}_{n\in\N}$ and arguing in a similar fashion yields
	$$
	C_{2,A,u}(\Omega)\geq \int_{\Omega} |\nabla\Im u(x)- A(x)\Re u(x)|dx,
	$$
	which, on account of \eqref{p-norm}, proves the opposite of inequality \eqref{pp}, concluding the proof of the first statement. If $u\in BV_A(\Omega)\cap C^\infty(\Omega)$, fix a compact set $K\subset\Omega$ with nonempty interior
	and consider 
	$$
	\tilde f:=f\chi_{{\rm int}(K)},\qquad 
	\tilde g:=g\chi_{{\rm int}(K)}.
	$$ 
	Then, as above, one can find two sequences	$\{\varphi_n\}_{n\in\N}$,$\{\psi_n\}_{n\in \N} \subset C^\infty_c({\rm int}(K),\R^N)$ such that $\varphi_n\to f$ and $\psi_n\to g$ pointwise  and $\|\varphi_n\|_{L^\infty({\rm int}(K))},\|\psi_n\|_{L^\infty({\rm int}(K))}\leq 1$, for all $n \in \N$. Then, we have
	\begin{align*}
	C_{1,A,u}(\Omega)& \geq \int_{\Omega}  \Re u(x) \dvg \varphi_n(x)-
	A(x)\cdot\varphi_n(x)\, \Im u(x) dx \\
	&=\int_{K}  \Re u(x) \dvg \varphi_n(x)-
	A(x)\cdot\varphi_n(x)\, \Im u(x) dx \\
	&=-\sum_{i=1}^N\int_{K} \big(\partial_{x_i}\Re u(x)
	+ A^{(i)}(x)\Im u(x) \big)\varphi^{(i)}_n(x)dx.
	\end{align*}
	Since $u\in C^\infty(\Omega)$, we have
	$\nabla \Re u+ A\Im u\in L^1(K)$. Thus, by the dominated convergence theorem,
	$$
	C_{1,A,u}(\Omega)\geq \int_{K} |\nabla\Re u(x)
	+ A(x)\Im u(x)|dx.
	$$
	The conclusion follows using an exhaustive sequence of compacts
	via monotone convergence.
\end{proof}

\smallskip
\noindent
We endow the space $BV_{A}(\Omega, \C)$ with the following norm:
$$
\|u\|_{BV_A(\Omega)} := \|u\|_{L^{1}(\Omega)} + |Du|_{A}(\Omega).
$$

\begin{lemma}[Norm equivalence]\label{Equiv}
	\label{PreExt}
	Let $\Omega \subset \R^N$ be an open and bounded set. 
	Let $A:\R^N \to \R^N$ be locally bounded.\
	Then $u \in BV_A(\Omega)$ if and only if $u \in BV(\Omega)$.
	Moreover, for every $u \in BV_A(\Omega)$, there exists a positive constant
	$K = K(A,\Omega)$ such that
	$$
	K^{-1} \|u\|_{BV(\Omega)} \leq \|u\|_{BV_{A}(\Omega)} \leq K \|u\|_{BV(\Omega)}.
	$$
	\begin{proof}
		Denoting by $\sup_{\varphi}$ the supremum over 
		functions  $\varphi\in C^\infty_c(\Omega,\R^N)$ with $\|\varphi\|_{L^\infty(\Omega)}\leq 1$, we get
		\begin{align*}
		|Du|(\Omega)&= |\Re u|(\Omega) + |\Im u|(\Omega) = 
		\sup_{\varphi}\int_{\Omega}\Re u(x) \, \dvg\,\varphi(x)\, dx 
		+ \sup_{\varphi}\int_{\Omega}\Im u(x) \, \dvg\,\varphi(x) \,dx \\
		&= \sup_{\varphi}\int_{\Omega} \Re u(x) \, \dvg\,\varphi(x) - \scal{A(x)}{\varphi(x)}\, \Im u(x) + \scal{A(x)}{\varphi(x)}\, \Im u \,dx \\
		&+
		\sup_{\varphi}\int_{\Omega} \Im u(x) \, \dvg\,\varphi(x) + \scal{A(x)}{\varphi(x)}\, \Re u(x) - \scal{A(x)}{\varphi(x)}\, \Re u(x) \, dx\\
		&\leq \sup_{\varphi}\int_{\Omega}\Re u(x) \, \dvg\,\varphi(x) - \scal{A(x)}{\varphi(x)} \Im u(x)\, dx + \sup_{\varphi}\int_{\Omega}\scal{A(x)}{\varphi(x)}\, \Im u(x)\,dx\\
		&+ \sup_{\varphi}\int_{\Omega} \Im u(x) \, \dvg\,\varphi(x) + \scal{A(x)}{\varphi(x)}\, \Re u(x)\, dx  + \sup_{\varphi}\int_{\Omega}\scal{A(x)}{(-\varphi)(x)}\, \Re u(x) \,dx \\
		& \leq C_{1,A,u}(\Omega) + C_{2,A,u}(\Omega) + \|A\|_{L^{\infty}(\Omega)} \|u\|_{L^{1}(\Omega)}.
		\end{align*}
		Therefore, we have that
		$$
		\|u\|_{BV(\Omega)} \leq (1+\|A\|_{L^{\infty}(\Omega)} ) \|u\|_{BV_{A}(\Omega)}.
		$$
		For the second inequality, we have
		\begin{align*}
		C_{1,A,u}(\Omega) &\leq \sup_{\varphi}\int_{\Omega}\Re u(x) \, \dvg\,\varphi(x)\,dx + \sup_{\varphi}\int_{\Omega}\scal{A(x)}{(-\varphi)(x)}\, \Im u(x) \,dx \\
		&\leq |D\Re u|(\Omega) + \|A\|_{L^{\infty}(\Omega)} \int_{\Omega}|\Im u|dx,
		\end{align*}
	and similarly for $C_{2,A,u}(\Omega)$.
		Therefore, we conclude
		$$
		\|u\|_{BV_{A}(\Omega)}\leq (1+\|A\|_{L^{\infty}(\Omega)} )\|u\|_{BV(\Omega)}.
		$$
		Calling $K := (1+\|A\|_{L^{\infty}(\Omega)}),$ concludes the proof.
	\end{proof}
\end{lemma} 

\begin{lemma}[Structure Theorem for $BV_{A}$ functions]\label{Struttura}
Let $\Omega \subset \R^N$ be an open and bounded set, $A:\R^N\to\R^N$ locally bounded and $u \in BV_{A}(\Omega)$.\ There exists a unique $\R^{2N}$-valued
finite Radon measure $\mu_{A,u} = (\mu_{1,A,u}, \mu_{2,A,u})$ such that 
\begin{equation*}
\begin{aligned}
\int_{\Omega}u(x) \dvg \varphi(x) + \i A(x)\cdot\varphi(x) \, u(x) \, dx &= \int_{\Omega}\Re u(x) \dvg \varphi(x) - A(x)\cdot \varphi(x)\, \Im u(x) \, dx\\
&\quad + \i \int_{\Omega} \Im u(x) \dvg \varphi(x) + A(x) \cdot \varphi(x) \Re u(x)\, dx \\
&= \int_{\Omega} \varphi(x)\cdot d(\mu_{1,A,u}+\i\mu_{2,A,u})(x), 
\end{aligned}
\end{equation*}
for every $\varphi\in C^{\infty}_c(\Omega, \R^N)$ and
\[
|Du|_A(\Omega)=|\mu_{1,A,u}|(\Omega)+ |\mu_{2,A,u}|(\Omega).
\]
\end{lemma}
\begin{proof}
Of course, we have 
\[
\left|\int_{\Omega}\Re u(x) \dvg \varphi(x) - A(x)\cdot \varphi(x)\, \Im u(x) \, dx\right|\leq C_{1,A,u}(\Omega) \|\varphi\|_{L^\infty(\Omega)}, \qquad \forall \varphi\in C^{\infty}_c(\Omega,\R^N).
\]
Then, a standard application of the Hahn-Banach theorem yields the existence
of a linear and continuous extension $L$ of the functional $\Psi:C^{\infty}_c(\Omega,\R^N)\to\R$
\[
\langle\Psi,\varphi\rangle= \int_{\Omega}\Re u(x) \dvg \varphi(x) - A(x)\cdot \varphi(x)\, \Im u(x) \, dx
\]
to the normed space $(C_c(\Omega, \R^N),\|\cdot\|_{L^\infty(\Omega)})$ such that 
$$
\|L\|=\|\Psi\|= C_{1,A,u}(\Omega).
$$
On the other hand, by the Riesz representation Theorem (cf.\ \cite[Corollary 1.55]{ABF})
there exists a unique $\R^N$-valued finite Radon measure $\mu_{1,A,u}$ with
\begin{align*}
L (\varphi) = \int_{\Omega}\varphi(x) \cdot d\mu_{1,A,u}(x),
\qquad \forall \varphi\in C_c(\Omega,\R^N),
\end{align*}
and such that $|\mu_{1,A,u}|(\Omega)=\|L\|$. Thus $|\mu_{1,A,u}|(\Omega)=C_{1,A,u}(\Omega)$.
The same argument can be repeated verbatim for the functional
\[\varphi \mapsto \int_{\Omega} \Im u(x) \dvg \varphi(x) + A(x) \cdot \varphi(x)\, \Re u(x) \, dx,\]
which concludes the proof.
\end{proof}

\begin{lemma}[Lower semicontinuity of $|Du|_A(\Omega)$]
	\label{semic}
	Let $A:\R^N\to\R^N$ be locally bounded.
Let $\Omega\subset\R^N$ be an open set and $\{u_k\}_{k\in\N}\subset BV_A(\Omega)$ 
a sequence converging locally in $L^1(\Omega)$ to a function $u$. Then
\[
\liminf_{k\to\infty} |Du_k|_A(\Omega)\geq |Du|_A(\Omega).
\]
\end{lemma}
\begin{proof}
Fix $\varphi\in C^\infty_c(\Omega,\R^N)$ with $\|\varphi\|_{L^\infty(\Omega)}\leq 1$. By the definitions of
$C_{i,A,u_k}(\Omega)$, we have
\begin{align*}
& C_{1,A,u_k}(\Omega)\geq \int_{\Omega}  \Re u_k(x) \dvg\varphi(x)-
A(x)\cdot\varphi(x) \Im u_k(x) dx, \\
& C_{2,A,u_k}(\Omega)\geq\int_{\Omega}  \Im u_k(x) \dvg\varphi(x)+
A(x)\cdot\varphi(x) \Re u_k(x) dx.
\end{align*}
By the convergence of $\{u_k\}_{k\in\N}$ in $L^1_{\loc}(\Omega,\C)$ to $u$, we get
\begin{align*}
& \liminf_{k\to\infty}C_{1,A,u_k}(\Omega)\geq \int_{\Omega}  \Re u(x) \dvg\varphi(x)-
A(x)\cdot\varphi(x) \Im u(x) dx, \\
& \liminf_{k\to\infty}C_{2,A,u_k}(\Omega)\geq\int_{\Omega}  \Im u(x) \dvg\varphi(x)+
A(x)\cdot\varphi(x) \Re u(x) dx.
\end{align*}
The assertion follows by the definition of $|Du|_A(\Omega)$ and the arbitrariness of such functions $\varphi$.
\end{proof}
\begin{lemma}
The space $(BV_A(\Omega),\|\cdot\|_{BV_A(\Omega)})$ is a real Banach space. 
\end{lemma}
\begin{proof}
	It is readily seen that $\|\cdot\|_{BV_A(\Omega)}$ is a norm (to this aim, it is enough to check that the map $u\mapsto |Du|_A(\Omega)$ defines a semi-norm
	over $BV_A(\Omega)$, which is left to the reader). Let us prove that the space is complete. Let $\{u_n\}_{n\in\N}\subset BV_A(\Omega)$ be a Cauchy sequence, namely for every $\eps>0$ there exists $n_0\in\N$ such that
	$$
	\int_{\Omega}|u_n-u_k|_1dx+|D(u_n-u_k)|_{A}(\Omega)<\eps,\quad
	\text{$\forall \,n,k\geq n_0$.}
	$$
	In particular, $\{u_n\}_{n\in\N}$ is a Cauchy sequence in the Banach space
	$(L^1(\Omega),\|\cdot\|_{L^1(\Omega)})$, which implies that there exists $u\in L^1(\Omega)$ with $\|u_n-u\|_{L^1(\Omega)}\to 0$, as $n\to\infty$. Therefore,
	in light of Lemma~\ref{semic}, we get 
	$$
	|D(u-u_k)|_{A}(\Omega)\leq \liminf_n |D(u_n-u_k)|_{A}(\Omega)\leq\eps,\quad
	\text{$\forall\,k\geq n_0$,}
	$$
	namely $|D(u_n-u)|_{A}(\Omega) \to 0$, as $n\to\infty$, which concludes the proof.
	\end{proof}
\begin{lemma}[Multiplication by Lipschitz functions]\label{form_LIP}
Let $\Omega\subset\R^N$ be an open set.
Let $A:\R^N\to\R^N$ be locally bounded and 
$u\in BV_{A, \mathrm{loc}}(\Omega)$. Then 
for every locally Lipschitz $\psi:\Omega\to \R$ the function $u\psi\in BV_{A, \mathrm{loc}}(\Omega)$ and 
\begin{align*}
\mu_{1,A,\psi u}&=\psi \mu_{1,A,u}- \Re u \cdot \nabla \psi \mathcal{L}^N, \\
\mu_{2,A,\psi u}&=\psi \mu_{2,A,u}- \Im u\cdot \nabla \psi \mathcal{L}^N
\end{align*}
where $\mathcal{L}^N$ denotes the $N-$dimensional Lebesgue measure.
\begin{proof}
Consider $U\Subset\Omega$ open and let $\varphi\in C^{\infty}_c(U,\R^N)$ be such that $\|\varphi\|_{L^\infty(U)}\leq 1$.
By Rademacher's theorem we have 
$\psi \dvg\varphi=\dvg(\psi \varphi)-\varphi\cdot \nabla\psi$ a.e. in $U$. Therefore, up to smoothing $\psi$, we get
\begin{equation*}
\begin{aligned}
& \int_{U}\Re (u\psi)(x) \, \dvg \varphi(x) - A(x) \cdot \varphi(x) \, \Im(u\psi)(x) dx \\
&=\int_{U}\psi(x) \Re u(x)\dvg\varphi(x)-A(x)\cdot \varphi(x) \psi(x) \Im u(x) dx\\
&=\int_{U}\Re u(x)\dvg(\psi\varphi)(x)-A(x)\cdot \varphi(x) \psi(x) \Im u(x) dx-\int_{U}\Re u(x)\varphi(x)\cdot \nabla\psi(x) dx\\
& \leq C_{1,A,u}(U)\|\psi\|_{L^{\infty}(\overline{U})}+\mathrm{Lip}(\psi)\|u\|_{L^{1}(U)}.
\end{aligned}\end{equation*}
A similar estimate holds for the second term,
proving $u \psi\in BV_{A,\mathrm{loc}}(\Omega)$. 
By Lemma \ref{Struttura}, we have
\begin{align*}
&\int_{\Omega}\varphi(x)\cdot d\mu_{1,A,u\psi}\\
&=\int_{\Omega}\psi(x) \Re u(x)\dvg\varphi(x)-A(x)\cdot \varphi(x) \psi(x) \Im u(x)\\
&=\int_{\Omega}\Re u(x)\dvg(\psi\varphi)(x)-A(x)\cdot \varphi(x) \psi(x) \Im u(x) dx-\int_{\Omega}\Re u(x)\varphi(x)\cdot \nabla\psi(x) dx\\
&=\int_{\Omega} \varphi(x) \psi(x) d\mu_{1,A,u}-\int_{\Omega}\Re u(x)\varphi(x)\cdot \nabla\psi(x) dx.
\end{align*}
and the thesis follows. A similar argument holds also for $\mu_{2,A,u\psi}$, and this concludes the proof.
\end{proof}
\end{lemma}

\noindent
Let $\eta\in C^{\infty}_0(\R^N)$ be a radial nonnegative 
function with $\int_{\R^N} \eta(x) dx=1$ and ${\rm supp}(\eta)\subset B_1(0)$. 
Given $\varepsilon>0$ and $u\in L^1(\Omega; \mathbb{C})$, 
extended to zero out of $\Omega$, we define the usual regularization
\begin{align}\label{mollif}
	u_\varepsilon(x):=\frac{1}{\varepsilon^N}\int_{\R^N}\eta\left(\frac{x-y}{\varepsilon}\right) u(y) dy=\frac{1}{\varepsilon^N}\int_{B(x,\eps)}\eta\left(\frac{x-y}{\varepsilon}\right) u(y) dy.
\end{align}

\noindent											
Next we have the magnetic counterpart of the classic
Anzellotti-Giaquinta Theorem \cite{AG78}.
\begin{lemma}[Approximation with smooth functions]\label{Approx}
	\label{approxim}
	Suppose that $A: \R^N \to \R^N$ is locally Lipschitz.
Let $\Omega \subset \R^N$ be an open and bounded set and let $u \in BV_{A}(\Omega)$.
Then there exists a sequence $\{u_k\}_{k \in \N} \subset C^{\infty}(\Omega, \mathbb{C})$ such that
$$
\lim_{k \to \infty} \int_{\Omega}|u_k - u|_1 \,dx =0
\quad\,\,\text{and}\quad\,\,
\lim_{k \to \infty}|Du_k|_A (\Omega) = |Du|_{A}(\Omega).
$$
\begin{proof}
We follow closely the proof of \cite[Theorem 5.3]{EG}.
In light of the semicontinuity property (Lemma~\ref{semic}),
it is enough to prove that, for every $\eps >0$, there exists a
function $v_{\eps} \in C^{\infty}(\Omega)$ such that
\begin{equation}\label{enough}
\int_{\Omega}|u-v_\eps|_1 dx < \eps, \quad 
\textrm{and} \quad |Dv_{\eps}|_A (\Omega) < |Du|_A (\Omega) + \eps.
\end{equation}
Let $\{\Omega_{j}\}_{j\in\N}$ be a sequence of domains 
defined, for  $m \in \N$, as follows
$$
\Omega_j := \left\{ x \in \Omega \,\,|\,\, \textrm{dist}(x,\partial \Omega) > \dfrac{1}{m+j}\right\} \cap B(0,k+m), \qquad j \in\N,
$$
\noindent where $B(0,k+m)$ denotes the open ball of center $0$ and radius $k+m$.\\
Since $|Du|_{A}$ is a Radon measure, given $\eps>0$ we can choose $m \in \N$ 
so large that
\begin{equation}\label{misura}
|Du|_{A}(\Omega \setminus \Omega_0) < \eps.
\end{equation}
We want to stress that the sequence of open domains $\{\Omega_{j}\}$ 
is built in such a way that
$$
\Omega_{j} \subset \Omega_{j+1} \subset \Omega, \quad \textrm{for any } j\in\N, \quad
\textrm{and} \quad \bigcup_{j=0}^\infty \Omega_{j} = \Omega.
$$
We now define another sequence of open domains $\{U_{j}\}_{j \in\N}$, by setting
$$
U_0 := \Omega_0, \qquad U_j := \Omega_{j+1} \setminus \overline{\Omega}_{j-1}, \quad \textrm{for $j \geq 1$}.
$$
By standard results, there exists a partition of unity related to the covering $\{U_j\}_{j\in\N}$,
which means that there exists $\{ f_{j} \}_{j\in\N} \in C^{\infty}_{c}(U_j)$ 
such that $0\leq f_j \leq 1$ for every $j \geq 0$ 
and $\sum_{j=0}^\infty f_j = 1$ on $\Omega$. 
We stress that the last property, in particular, implies that 
\begin{equation}\label{partition}
\sum_{j=0}^{\infty} \nabla f_j = 0, \quad\,\, \textrm{on $\Omega$}.
\end{equation}
Recalling the definition of the norm $|\cdot|_1$
given by  \eqref{p-norm},  
and the classical properties of the convolution, we 
easily get that for every $j \geq 0$ there
exists $0<\eps_j<\eps$ such that
\begin{equation}\label{epsj}
\left\{ \begin{array}{l}
           \textrm{supp}\left( ( f_j u)_{\eps_j}\right) \subset U_j,\\ 
           \noalign{\vskip3pt}
					 \displaystyle\int_{\Omega}\left| (f_j u)_{\eps_j} - f_j u \right|_1 dx < \eps \, 2^{-(j+1)},\\
					 \noalign{\vskip6pt}
					 \displaystyle\int_{\Omega}\left| (u \nabla f_j)_{\eps_j} - u \nabla f_j \right|_1 dx < \eps \, 2^{-(j+1)}.
				\end{array}\right.
\end{equation}
We can now define
$v_\eps := \sum_{j=0}^{\infty}(u f_j)_{\eps_j}.$
Since the sum is locally finite, 
we have that $v_{\eps} \in C^{\infty}(\Omega, \C)$, and
that $u = \sum_{j=0}^\infty u f_j$ pointwise.					
Let us start considering the real part of the linear functional 
\[
	 C^{\infty}_c(\Omega)\ni\varphi\mapsto \int_{\Omega}v_{\eps}(x) \dvg \varphi (x) +\i A(x)\cdot \varphi(x)\, v_{\eps}(x)\, dx.
\]
We have
\begin{equation*}
\begin{aligned}
\int_{\Omega}&\Re v_{\eps}(x) \dvg \varphi(x) - A(x)\cdot \varphi(x)\Im v_{\eps}(x)\, dx \\
&= \sum_{j=0}^{\infty}\int_{\Omega}\left( (\Re u f_j)\ast \eta_{\eps_j} \right)(x) \dvg \varphi(x) - \sum_{j=0}^{\infty}\int_{\Omega}A(x)\cdot \varphi(x) \left((\Im u f_j)\ast \eta_{\eps_j}\right)(x)\, dx
=: \mathcal{I} - \mathcal{II}.
\end{aligned}\end{equation*}
Now
\begin{equation*}
\begin{aligned}
\mathcal{I}&= \sum_{j=0}^{\infty}\dfrac{1}{\eps_j^N}\int_{\Omega}\int_{\Omega}\Re u(y)f_j(y) \eta\left(\dfrac{x-y}{\eps_j}\right) \dvg \varphi(x)\, dy dx 
=\sum_{j=0}^{\infty}\int_{\Omega}\Re u (y) f_j(y) \dvg (\varphi \ast \eta_{\eps_j})(y) \, dy \\
&=\sum_{j=0}^{\infty}\int_{\Omega}\Re u(y) \dvg \left( f_j \,(\varphi \ast \eta_{\eps_j})\right)(y) \, dy - \sum_{j=0}^{\infty}\int_{\Omega}\Re u(y) \nabla f_j(y)\cdot (\varphi \ast \eta_{\eps_j})(y) \, dy\\
&=\sum_{j=0}^{\infty}\int_{\Omega}\Re u(y) \dvg \left( f_j \,(\varphi \ast \eta_{\eps_j})\right)(y) \, dy - \sum_{j=0}^{\infty}\int_{\Omega}\left[ \left((\Re u \nabla f_j)\ast \eta_{\eps_j}\right)(y) - \Re u (y) \nabla f_j(y)\right] \cdot \varphi(y)\, dy\\
&=: \mathcal{I}' - \mathcal{I}'',
\end{aligned}\end{equation*}
where in the last equality we used \eqref{partition}.
For $\mathcal{II}$, we have
\begin{equation*}
\begin{aligned}
\mathcal{II}&= \sum_{j=0}^{\infty}\int_{\Omega}A(x)\cdot \varphi(x) \left[ \dfrac{1}{\eps_j^N}\int_{\Omega}\Im u(y)f_j(y)\eta\left(\dfrac{x-y}{\eps_j}\right) \, dy\right] \, dx \\
&= \sum_{j=0}^{\infty}\dfrac{1}{\eps_j^N}\int_{\Omega}\int_\Omega A(y)\cdot \varphi(x)\Im u(y)f_j(y)\eta\left(\dfrac{x-y}{\eps_j}\right) \, dxdy \\
&\quad + \sum_{j=0}^{\infty}\dfrac{1}{\eps_j^N}\int_{\Omega}\int_{\Omega}(A(x)-A(y))\cdot \varphi(x)\Im u(y)f_j(y)\eta\left(\dfrac{x-y}{\eps_j}\right) \, dxdy \\
&= \sum_{j=0}^{\infty}\int_{\Omega}A(y)\cdot \left(f_j (\varphi \ast \eta_{\eps_j}) \right)(y) \Im u (y) \, dy \\
&\quad+ \sum_{j=0}^{\infty}\dfrac{1}{\eps_j^N}\int_{\Omega}\int_{\Omega}(A(x)-A(y))\cdot \varphi(x)\Im u(y)f_j(y)\eta\left(\dfrac{x-y}{\eps_j}\right) dxdy. 
\end{aligned}\end{equation*}
Denoting $f_j (\varphi \ast \eta_{\eps_j}) := \left( f_j (\varphi_{1} \ast \eta_{\eps_j}), \ldots, f_j (\varphi_{n}\ast \eta_{\eps_j})\right),$ 
we note that $|f_j (\varphi \ast \eta_{\eps_j})|\leq 1$ for any $ j\geq 0$,
whenever $\|\varphi\|_{L^\infty(\Omega)}\leq 1$. We also stress that 
$|\mathcal{I}''| < \eps,$
because of \eqref{epsj}.
Therefore, 
\begin{equation}\label{Approx2}
\begin{aligned}
&\Big|\int_{\Omega} \Re v_{\eps}(x) \dvg \varphi(x) - A(x)\cdot \varphi(x)\Im v_{\eps}(x)\, dx \Big| \\
&\leq \left| \sum_{j=0}^{\infty}\int_{\Omega}\Re u(y) \dvg \left( f_j \,(\varphi \ast \eta_{\eps_j})\right)(y) - A(y)\cdot \left(f_j (\varphi \ast \eta_{\eps_j}) \right)(y) \Im u (y) \, dy \right| \\
&+ \sum_{j=0}^{\infty}\left|\dfrac{1}{\eps_j^N}\int_{\Omega}\int_{\Omega}(A(x)-A(y))\cdot \varphi(x)\Im u(y)f_j(y)\eta\left(\dfrac{x-y}{\eps_j}\right)\, dxdy\right| + \eps.
\end{aligned}
\end{equation}
Now, 
$$
\left| \sum_{j=0}^{\infty}\int_{\Omega}\Re u(y) \dvg \left( f_j \,(\varphi \ast \eta_{\eps_j})\right)(y) - A(y)\cdot \left(f_j (\varphi \ast \eta_{\eps_j}) \right)(y) \Im u (y) \, dy \right|
$$ 
can be treated as in \cite[Theorem 2, Section 5.2.2.]{EG}. Indeed,
recalling that by construction every point $x \in \Omega$ belongs to at most
three of the sets $U_j$, we have
\begin{equation*}
\begin{aligned}
\Big| &\sum_{j=0}^{\infty}\int_{\Omega}\Re u(y) \dvg \left( f_j \,(\varphi \ast \eta_{\eps_j})\right)(y) - A(y)\cdot \left(f_j (\varphi \ast \eta_{\eps_j}) \right)(y) \Im u (y) \, dy \Big| \\
&=\Big|\int_{\Omega}\Re u(y) \dvg \left( f_0 \,(\varphi \ast \eta_{\eps_0})\right)(y) - A(y)\cdot \left(f_0 (\varphi \ast \eta_{\eps_0}) \right)(y) \Im u (y) \, dy \\
&+ \sum_{j=1}^{\infty}\int_{\Omega}\Re u(y) \dvg \left( f_j \,(\varphi \ast \eta_{\eps_j})\right)(y) - A(y)\cdot \left(f_j (\varphi \ast \eta_{\eps_j}) \right)(y) \Im u (y) \, dy \Big|\\
&\leq C_{1,A,u}(\Omega) + \sum_{j=1}^{\infty}C_{1,A,u}(U_j) \leq C_{1,A,u}(\Omega) + 3 \, C_{1,A,u}(\Omega \setminus \Omega_0)\\
&\leq C_{1,A,u}(\Omega) + 3 \, \eps,
\end{aligned}\end{equation*}
where the last inequality follows from \eqref{misura}.
It remains to estimate
$$
\sum_{j=0}^{\infty}\left|\dfrac{1}{\eps_j^N}\int_{\Omega}\int_{\Omega}(A(x)-A(y))\cdot \varphi(x)\Im u(y)f_j(y)\eta\left(\dfrac{x-y}{\eps_j}\right) dxdy\right|=: \sum_{j=0}^{\infty}|\mathcal{III}_j|.
$$
Recalling that $A$ is locally Lipschitz, 
$\|\varphi\|_{L^\infty(\Omega)}\leq 1$ and that $\textrm{supp}(\eta) \subset B_{1}(0),$ we have  
\begin{equation*}
\begin{aligned}
\sum_{j=0}^{\infty}|\mathcal{III}_j| &\leq \textrm{Lip}(A,\Omega) \eps \, \int_{\R^N}\eta(z)dz \int_{\Omega}\sum_{j=0}^{\infty}f_j(y)|\Im u (y)| \, dy \\
&= \eps \, \textrm{Lip}(A,\Omega) \, \|\Im (u)\|_{L^1(\Omega)} =: C \, \eps.
\end{aligned}
\end{equation*}
Going back to \eqref{Approx2}, taking the supremum over $\varphi$ and by the
arbitrariness of $\eps >0$ we get precisely \eqref{enough} for the real part.
An analogous argument provides \eqref{enough} also for the imaginary part and 
this concludes the proof.
\end{proof}
\end{lemma}

\begin{definition}[Extension domains]\label{exdomain}
	Let $A:\R^N\to\R^N$ be a locally bounded function.
Let $\Omega \subset \R^N$ be an open set. We say that $\Omega$ is
an extension domain if its boundary $\partial \Omega$ is bounded
and for any open set $W\supset \overline{\Omega}$, there exists a
linear and continuous extension operator $E: BV_A(\Omega) \to BV_A(\R^N)$
such that
\begin{equation*}
Eu = 0, \quad \textrm{for almost every } x \in \R^N \setminus W, \quad \textrm{and} \quad |DEu|_A(\partial \Omega) = 0,
\end{equation*}
for every $u \in BV_A(\Omega)$.
\end{definition}

\begin{lemma}[Lipschitz extension domains]
	\label{ExtDom-new}
	Let $\Omega \subset \R^N$ be an open bounded set with  
	Lipschitz boundary and $A:\R^N\to\R^N$
	locally Lipschitz. Then $\Omega$ is an extension domain.
	\begin{proof}
Given an arbitrary open set $W\supset \overline{\Omega}$, by virtue of 
\cite[Proposition 3.21]{ABF} there exists a
 linear and continuous extension operator 
 $E_0: BV(\Omega,\R) \to BV(\R^N,\R),$
 such that
 \begin{equation*}
 E_0u = 0, \quad \textrm{for almost every } x \in \R^N \setminus W, \quad \textrm{and} \quad |DE_0u|(\partial \Omega) = 0,
 \end{equation*}
 for all $u \in BV(\Omega)$. Given $u\in BV_A(\Omega)$, we have from Lemma~\ref{Equiv} that 
 $u\in BV(\Omega)$, which means that both $\Re u$ and $\Im u$
 are elements of $BV(\Omega,\R)$. Let us define
 $$
 Eu:=E_0\Re u+\i E_0\Im u,\quad\,\, u\in BV_A(\Omega).
 $$ 
Then $|DE_0\Re u|(\partial \Omega) =|DE_0\Im u|(\partial \Omega) = 0$ 
and there exists a positive constant $C_W$ depending on $W$ and $\Omega$ with
$$
\|E_0\Re u\|_{BV(\R^N)}\leq C_W\|\Re u\|_{BV(\Omega)},\qquad
\|E_0\Im u\|_{BV(\R^N)}\leq C_W\|\Im u\|_{BV(\Omega)}.
$$
 Taking into account Lemma~\ref{Equiv}, we have that
 \begin{align*}
 \|Eu\|_{BV_A(\R^N)}&=C_{1,A,Eu}(\R^N)+C_{2,A,Eu}(\R^N)+\|E_0 \Re u\|_{L^1(\R^N)} +\|E_0 \Im u\|_{L^1(\R^N)}\\
& \leq |D E_0\Re u|(\R^N)+\|A\|_{L^\infty(W)}\|E_0\Im u\|_{L^1(\R^N)} +\|E_0 \Re u\|_{L^1(\R^N)} +\|E_0 \Im u\|_{L^1(\R^N)}\\
&+ |D E_0\Im u|(\R^N)+\|A\|_{L^\infty(W)}\|E_0\Re u\|_{L^1(\R^N)} +\\
&\leq (1+\|A\|_{L^\infty(W)})(\|E_0\Re u\|_{BV(\R^N)}+\|E_0 \Im u\|_{BV(\R^N)}) \\
&\leq (1+\|A\|_{L^\infty(W)})C_W(\|\Re u\|_{BV(\Omega)}+\|\Im u\|_{BV(\Omega)}) \\
& = (1+\|A\|_{L^\infty(W)})C_W\|u\|_{BV(\Omega)} \\
& \leq (1+\|A\|_{L^\infty(W)})C_W K \|u\|_{BV_A(\Omega)}.
 \end{align*}
 Therefore, there exists $C=C(A,\Omega,W)>0$ such that
 $$
 \|Eu\|_{BV_A(\R^N)}\leq C\|u\|_{BV_A(\Omega)},\quad \text{for all $u\in BV_A(\Omega)$.}
 $$
 We have to prove that $|DEu|_A(\partial \Omega) = 0$. We have
 	\begin{align*}
 		|DEu|_A(\partial \Omega):=\inf\{C_{1,A,Eu}(U)\ |\ \partial \Omega\subset U,\ U\, \mbox{open}\}+\inf\{C_{2,A,Eu}(U)\ |\ \partial \Omega\subset U\, \mbox{open}\}.
 	\end{align*}
Then, for arbitrary $U,U',U''$ open with 
$\partial\Omega\subset U\subset U'\subset U''\subset W$, we have
 \begin{align*}
 |DEu|_A(\partial \Omega)\leq |DEu|_A(U) &\leq |DE_0\Re u|(U)+|DE_0\Im u|(U)
+\|A\|_{L^\infty(W)}\|Eu\|_{L^1(U)} \\
&\leq |DE_0\Re u|(U)+|DE_0\Im u|(U')
+\|A\|_{L^\infty(W)}\|Eu\|_{L^1(U'')}.
 \end{align*}
 Taking the infimum over $U$ and recalling that $|DE_0\Re u|(\partial \Omega) = 0$ yields
 $$
 |DEu|_A(\partial \Omega)\leq |DE_0\Im u|(U')
 +\|A\|_{L^\infty(W)}\|Eu\|_{L^1(U'')}.
 $$
 Taking the infimum over $U'$ and recalling that $|DE_0\Im u|(\partial \Omega) = 0$ yields
 $$
 |DEu|_A(\partial \Omega)\leq \|A\|_{L^\infty(W)}\|Eu\|_{L^1(U'')}.
 $$
 Finally, taking as $U''$ a sequence $\{U''_j\}_{j\in \N}$ of open sets such
that $\partial \Omega \subset U''_j \subset W$ and
with ${\mathcal L}^N(U''_j)\to 0$ as $j\to \infty$,
 we conclude that $|DEu|_A(\partial \Omega)=0$.
\end{proof}
\end{lemma}

\begin{lemma}[Convolution]\label{conv}
Assume that $A:\R^N\to\R^N$ is locally Lipschitz. 
Suppose $U\subset \R^N$ is an open set with $U\Subset\Omega$ and let $u\in BV_A(\Omega)$. 
Then, for every sufficiently small $\varepsilon>0$, there holds
\[
|Du_{\varepsilon}|_{A}(U)\leq |Du|_{A}(\Omega)+ \varepsilon {\rm Lip}(A,\Omega) \|u\|_{L^1(\Omega)}.
\]
\begin{proof}
Fix $\varphi\in C^1_c(U,\R^N)$ with $\|\varphi\|_{L^\infty(U)}\leq 1$. 
Choose $\delta>0$ such that $\{x\in\R^N |\ d(x,U)<\delta\}\subset \Omega$. 
Then we have $\|\varphi_\eps\|_{L^\infty(\Omega)}\leq 1$ 
and ${\rm supp}(\varphi_{\varepsilon}) \subset \{x\in\R^N |\ d(x, U)<\delta\}$ 
for all small $\varepsilon>0$. Then
\begin{align*}
&\int_{U} \Re u_{\varepsilon}(x) \dvg\,\varphi(x)-A(x)\cdot \varphi(x)\Im u_{\varepsilon}(x) dx \\
&=\int_{\Omega} \left(\Re u\right)_{\varepsilon}(x) \dvg\,\varphi(x)-A(x)\cdot \varphi(x)\left(\Im u\right)_{\varepsilon}(x) dx\\
&=\int_{\Omega} \Re u(x) (\dvg\,\varphi)_{\varepsilon}(x)-\left(A(x)\cdot \varphi(x)\right)_{\varepsilon}\Im u(x) dx\\
&=\int_{\Omega} \Re u(x) \dvg\,\varphi_{\varepsilon}(x)-A(x)\cdot \varphi_{\varepsilon}(x)\Im u(x) dx\\
&-\int_{\Omega}\frac{1}{\varepsilon^N} \int_{\R^N} \eta\left(\frac{x-y}{\varepsilon}\right)(A(y)-A(x))\cdot \varphi(y) dy \Im u(x) dx\\
&\leq \int_{\Omega} \Re u(x) \dvg\,\varphi_{\varepsilon}(x)-A(x)\cdot \varphi_{\varepsilon}(x)\Im u(x) dx\\
&+\int_{\Omega}\frac{1}{\varepsilon^N} \int_{B(x,\varepsilon)} \eta\left(\frac{x-y}{\varepsilon}\right)\left|A(y)-A(x)\right| dy \left|\Im u(x)\right| dx\\
&\leq C_{1,A,u}(\Omega)+\eps {\rm Lip}(A,\Omega) \|u\|_{L^1(\Omega)}.
\end{align*}
Similarly, for every $\varphi\in C^1_c(U,\R^N)$ with $\|\varphi\|_{L^\infty(U)}\leq 1$,  we get
\begin{equation*}
\int_{U} \Im u_{\varepsilon}(x) \dvg\varphi(x)+A(x)\cdot \varphi(x)\Re u_{\varepsilon}(x) dx
\leq C_{2,A,u}(\Omega)+\eps {\rm Lip}(A,\Omega) \|u\|_{L^1(\Omega)}.
\end{equation*}
By the definition of $|Du|_{A}(\Omega)$ and taking the supremum over all $\varphi$ we get the assertion.
\end{proof}\end{lemma}

\begin{lemma}[Compactness for $BV_A(\Omega)$ functions]\label{compactness}
	Assume that $\Omega \subset \R^N$ is a bounded domain with Lipschitz boundary
	and that $A:\R^N\to\R^N$ is locally bounded. 
	Let $\{u_k\}_{k\in\N}$ be a bounded sequence in $BV_A(\Omega)$. Then,
	up to a subsequence, it converges strongly in $L^1(\Omega)$ to some function $u\in BV_A(\Omega)$.
\end{lemma}
\begin{proof}
By the approximation Lemma~\ref{approxim}, for any $k\in\N$ there is
$v_k\in BV_A(\Omega)\cap C^\infty(\Omega)$ such that
\begin{equation}
\label{controX}
\int_{\Omega}|u_k-v_k|_1dx<\frac{1}{k},\qquad
\sup_{k\in\N} |Dv_k|_A(\Omega)=C,
\end{equation}
for some $C>0$. In particular, we have
$$
\int_{\Omega}|v_k|_1dx\leq \int_{\Omega}|u_k-v_k|_1dx
+\int_{\Omega}|u_k|_1dx\leq C'+1,\qquad\,\,
C':=\sup_{k\in\N}\|u_k\|_{L^1(\Omega)}.
$$
Now, Lemma~\ref{Bv-lemmino} yields $v_k\in W^{1,1}_A(\Omega)$ and 
$$
\int_{\Omega}|\nabla v_k-\i Av_k|_1dx= |Dv_k|_A(\Omega).
$$
Therefore, we obtain
\begin{align*}
\int_{\Omega}|\nabla v_k|_1dx &\leq \int_{\Omega}|\nabla v_k-\i Av_k|_1dx+C_1\int_{\Omega}|Av_k|_1dx \\
&\leq |Dv_k|_A(\Omega)
+C_1\|A\|_{L^\infty(\overline{\Omega})}\|v_k\|_{L^1(\Omega)}\leq C'',
\end{align*}
for some $C''>0$.  Hence we infer that $\{v_k\}_{k\in\N}$ is a bounded sequence in $W^{1,1}(\Omega)$. Since $\partial\Omega$ is smooth, from Rellich compact embedding theorem there exists a subsequence $\{v_{k_j}\}_{j\in\N}$ of $\{v_k\}_{k\in\N}$ and $w\in L^{1}(\Omega)$ such that $v_{k_j}\to w$ in $L^{1}(\Omega)$. Then
from \eqref{controX} we get $u_{k_j}\to w$ in $L^1(\Omega)$. By the semi-continuity Lemma~\ref{semic} we obtain 
$$
|Dw|_A(\Omega)\leq \liminf_{k_j} |Dv_{k_j}|_A(\Omega)\leq C,
$$
which shows that $w\in BV_A(\Omega)$ and concludes the proof.
\end{proof}

\section{Proof of the main result}\label{main_res}
%

\noindent

%

\noindent
We now state two results that will be proven in the next section.
In the following $Q_{p,N}$ is as in definition \eqref{valoreK}. 

\begin{theorem}[$BV_A$-case]
	\label{Main}
Let $\Omega\subset\R^N$ be an open bounded set with Lipschitz boundary and $A:\R^N\to \R^N$ of class $C^2$. Let $u\in BV_A(\Omega)$ 
and consider a sequence $\{\rho_m\}_{m\in\mathbb{N}}$ of non-negative radial functions with
\begin{equation}
\label{lim-cond}
\lim_{m\to \infty} \int_{0}^{\infty} \rho_m(r) r^{N-1} dr=1,
\end{equation}
and such that, for every $\delta>0$,
\begin{equation}
\label{prorho}
\lim_{m\to \infty}\int_{\delta}^{\infty} \rho_m(r) r^{N-1} dr=0.
\end{equation}
Then, we have
\[
\lim_{m\to \infty}\int_{\Omega}\int_{\Omega} \frac{|u(x)-e^{\i(x-y)\cdot A(\frac{x+y}{2})}u(y)|_1}{|x-y|} \rho_m(x-y) dxdy= Q_{1,N} |D u|_{A}(\Omega).
\]
\end{theorem}

\begin{theorem}[$W^{1,p}_A(\Omega)$ case]
	\label{Main2}
	Let $\Omega\subset\R^N$ be an open bounded set with Lipschitz boundary and $A\in C^2(\R^N, \R^N)$. 
	Let $p\geq 1$, $u\in W^{1,p}_A(\Omega)$ 
	and $\{\rho_m\}_{m\in\mathbb{N}}$ as in Theorem~\ref{Main}. Then, we have
	\[
	\lim_{m\to \infty}\int_{\Omega}\int_{\Omega} \frac{|u(x)-e^{\i(x-y)\cdot A(\frac{x+y}{2})}u(y)|_p^p}{|x-y|^p} \rho_m(x-y) dxdy= p\, Q_{p,N} 
	\int_\Omega |\nabla u-\i Au|^p_pdx.
	\]
\end{theorem}

\begin{remark}\rm
	\label{remconv}
In the notation of Theorem~\ref{Main}, assuming \eqref{lim-cond} and \eqref{prorho}
automatically implies that
\begin{equation*}
\lim_{m\to \infty} \int_{0}^{\delta} \rho_m(r) r^{N-1+\beta} dr=0,\quad\text{for every $\beta>0$ and for every $\delta>0$}.
\end{equation*}
In fact, fixed $\delta>0$, taking an arbitrary $0<\tau<\delta$, we have
\begin{align*}
 \int_{0}^{\delta}\rho_{m}(r)r^{N-1+\beta}dr &=\int_{0}^{\tau}\rho_{m}(r)r^{N-1+\beta}dr+\int_{\tau}^{\delta}\rho_{m}(r)r^{N-1+\beta}dr \\
 &\leq
\tau^\beta\int_{0}^{\tau}\rho_{m}(r)r^{N-1}dr+\delta^\beta\int_{\tau}^{\infty}\rho_{m}(r)r^{N-1}dr
\leq C\tau^\beta+\delta^\beta\int_{\tau}^{\infty}\rho_{m}(r)r^{N-1}dr,
\end{align*}
from which the assertion follows by letting $m\to \infty$ first, using \eqref{prorho}, 
and finally letting $\tau\searrow 0$.
\end{remark}

\vskip5pt
\noindent
$\bullet$ {\bf Proof of the main result (Theorem~\ref{main}) completed.}
Let $r_\Omega$ denote the diameter of $\Omega$. Then we consider a function 
$\psi\in C^\infty_c(\R^N)$, $\psi(x)=\psi_0(|x|)$ with $\psi_0(t)=1$ for $t<r_\Omega$ 
and $\psi_0(t)=0$ for $t>2r_\Omega$. Then $\psi_0(|x-y|)=1,$ for every $x,y\in\Omega$.
Let $\{s_m\}_{m\in\N}\subset (0,1)$ with $s_m\nearrow 1$. 
For a $p\geq 1$ consider the sequence of radial functions in $L^1(\R^N)$
\begin{equation}
\label{def-rho}
\rho_m(|x|):=\frac{p(1-s_m)}{|x|^{N+ps_m-p}}\psi_0(|x|),
\,\,\quad x\in\R^N,\,\, m\in\N.
\end{equation}
Notice that both conditions \eqref{lim-cond} and \eqref{prorho} hold, since
\begin{equation*}
\lim_{m\to \infty}\int_{0}^{r_\Omega}\rho_m(r)r^{N-1}dr
=\lim_{m\to \infty}p(1-s_m)\int_{0}^{r_\Omega} r^{-ps_m+p-1}dr=\lim_{m\to \infty} r_\Omega^{p(1-s_m)}=1,
\end{equation*}
and
$$
\lim_{m\to \infty}\int_{r_\Omega}^{2 r_\Omega}\rho_m(r) r^{N-1}dr
=\lim_{m\to \infty}p(1-s_m)\int_{r_\Omega}^{2r_\Omega}\frac{\psi_0(r)}{r^{ps_m+1-p}}dr
\leq C\lim_{m\to \infty} 1-s_m=0.
$$
In a similar fashion, for any $\delta>0$, there holds
\begin{equation*}
\lim_{m\to \infty}\int_{\delta}^{\infty}\rho_{m}(r)r^{N-1}dr
\leq C\, \lim_{m\to \infty} p(1-s_m)\int_{\delta}^{2r_\Omega}\frac{1}{t^{ps_m+1-p}}dt=0.
\end{equation*}
Then Theorem~\ref{main} follows directly from Theorems~\ref{Main} and \ref{Main2} using $\rho_m$ as in \eqref{def-rho}. 
\qed

We first need the following

\begin{lemma}\label{rem}
Let $p\geq 1$. 
Then, for every $v\in\C^N$ it holds
\begin{align}\label{complex}
\lim_{m\to \infty}\int_{\R^N}\left|v\cdot \frac{h}{|h|}\right|^p_p\rho_m(h) dh=pQ_{p,N} |v|^p_p.
\end{align}
\end{lemma}


\begin{proof}
First of all we observe that, due to symmetry reasons, $Q_{p,N}$ is
{\em independent} of the choice of the direction $\boldsymbol\omega\in {\mathbb S}^{N-1}$. 	
We prove that \eqref{complex} easily follows assuming \eqref{complex} with $v\in\R^N$. 
Let $v=(v_1,\ldots, v_N)\in\C^N$ and $h=(h_1,\ldots, h_N)\in\R^N$. Then
\begin{align}\label{prim}
\left|v\cdot \frac{h}{|h|}\right|^p_p &=\left|\sum_{j=1}^N v_j \frac{h_j}{|h|}\right|^p_p=\left| \sum_{j=1}^N \Re v_j \frac{h_j}{|h|}+\i\sum_{j=1}^N \Im v_j \frac{h_j}{|h|}\right|^p_p\\
\nonumber
&=\left|\sum_{j=1}^N \Re v_j \frac{h_j}{|h|}\right|^p+\left|\sum_{j=1}^N \Im v_j \frac{h_j}{|h|}\right|^p=\left| \Re v\cdot \frac{h}{|h|}\right|^p+ \left| \Im v\cdot \frac{h}{|h|}\right|^p,
\end{align}
where we denoted by $\Re v=(\Re v_1,\ldots, \Re v_N)$ and $\Im v=(\Im v_1,\ldots, \Im v_N)$. Using \eqref{prim} we get
\begin{align*}
\lim_{m\to \infty}\int_{\R^N}\left|v\cdot \frac{h}{|h|}\right|^p_p\rho_m(h) dh &=\lim_{m\to \infty}\int_{\R^N}\left| \Re v\cdot \frac{h}{|h|}\right|^p\rho_m(h) dh+ \lim_{m\to \infty}\int_{\R^N}\left| \Im v\cdot \frac{h}{|h|}\right|^p\rho_m(h) dh\\
&=pQ_{p,N}\left( \left|\Re v\right|^p+ \left|\Im v\right|^p\right)=pQ_{p,N} |v|^p_p.
\end{align*}
In order to prove \eqref{complex} with $v\in\R^N$, we apply co-area formula, a change of variable and \eqref{lim-cond}, getting
\begin{align*}
\lim_{m\to \infty}&\int_{\R^N}\left|v\cdot \frac{h}{|h|}\right|^p\rho_m(h) dh=\lim_{m\to \infty}\int_0^{\infty}\int_{\{|h|=R\}}\left|v\cdot \frac{h}{|h|}\right|^p\rho_m(h) d\mathcal{H}^{N-1}(h) dR\\
\nonumber
&=\lim_{m\to \infty}\int_0^{\infty}\rho_m(R) R^{N-1}dR\int_{\mathbb{S}^{N-1}}\left|v\cdot  h\right|^pd\mathcal{H}^{N-1}(h) \\
\nonumber
&=|v|^p\int_{\mathbb{S}^{N-1}}\left|\frac{v}{|v|}\cdot h\right|^p d\mathcal{H}^{N-1}(h)
=|v|^p\int_{\mathbb{S}^{N-1}}\left|\boldsymbol\omega\cdot h\right|^p d\mathcal{H}^{N-1}(h)=pQ_{p,N} |v|^p,
\end{align*}
for an arbitrarily fixed $\boldsymbol\omega\in {\mathbb S}^{N-1}$. This concludes the proof.
\end{proof}


\noindent
Let now  $\{\rho_m\}_{m\in\mathbb{N}}$ be as in Theorem \ref{Main}. 
The following is the main result for smooth functions.



\begin{proposition}[Smooth case]\label{form-smooth}
	Let $\Omega\subset\R^N$ be a bounded set and $A\in C^2(\R^N,\R^N)$. Then 
	\[
	\lim_{m\to \infty}\int_{\Omega}\int_{\Omega} \frac{|u(x)-e^{\i(x-y)\cdot A(\frac{x+y}{2})}u(y)|^p_p}{|x-y|^p}\rho_m(x-y) dxdy= 
	pQ_{p,N}\int_{\Omega} |\nabla u-\i Au|_p^pdx,
	\]
	for every $u\in C^2(\bar\Omega,\mathbb{C})$ and for every $p\geq 1$. In particular, if $p=1$ then
	\begin{align}\label{Pociarellosmooth}
	\lim_{m\to \infty}\int_{\Omega}\int_{\Omega} \frac{|u(x)-e^{\i(x-y)\cdot A(\frac{x+y}{2})}u(y)|_1}{|x-y|}\rho_m(x-y) dxdy= 
	Q_{1,N}|Du|_A(\Omega).
	\end{align}
\end{proposition}
\begin{proof}
Let $p\geq 1$. If we set $\varphi(y):=e^{\i(x-y)\cdot A(\frac{x+y}{2})} u(y)$, since
\[
\nabla_y\varphi(y)=e^{\i(x-y)\cdot A(\frac{x+y}{2})}
\Big(\nabla_y u(y)-\i A\Big(\frac{x+y}{2}\Big)u(y)+\frac{\i}{2}u(y)(x-y)\cdot \nabla_y A\Big(\frac{x+y}{2}\Big)\Big),
\]
if $x,y\in\Omega$, since 
$u,A\in C^2(\bar{\Omega})$, by Taylor's formula we get (for $y\in B(x,\rho)\subset\Omega$)
\[
\frac{u(x)-e^{\i(x-y)\cdot A(\frac{x+y}{2})} u(y)}{|x-y|}=\frac{\varphi(x)-\varphi(y)}{|x-y|}
=(\nabla u(x)-\i A(x)u(x))\cdot \frac{x-y}{|x-y|}+{\mathcal O}(|x-y|).
\]
Then, taking into account $(ii)$ of Lemma~\ref{propnorm} below, applied with 
$T(x):=\nabla u(x)-\i A(x)u(x)$
we get
\[
\Big|\frac{u(x)-e^{\i(x-y)\cdot A(\frac{x+y}{2})} u(y)}{|x-y|}\Big|_p^p=\Big|(\nabla u(x)-\i A(x)u(x))\cdot \frac{x-y}{|x-y|}\Big|_p^p+{\mathcal O}(|x-y|).
\]
For $x\in\Omega$, if we set $R_x={\rm dist}(x,\partial\Omega)$, then we get for some positive constant $C$
\begin{align*}
\Psi_m(x)&:=\int_{\Omega} \Big|\frac{|u(x)-e^{\i(x-y)\cdot A(\frac{x+y}{2})}u(y)|_p^p-|(\nabla u(x)-\i A(x)u(x))\cdot (x-y)|_p^p}{|x-y|^p}\rho_{m}(x-y) \Big|dy \\
&=\int_{B(x,R_x)} \Big|
\Big|\frac{u(x)-e^{\i(x-y)\cdot A(\frac{x+y}{2})} u(y)}{|x-y|}\Big|_p^p-\Big|(\nabla u(x)-\i A(x)u(x))\cdot \frac{x-y}{|x-y|}\Big|_p^p
\Big|\rho_{m}(x-y) dy\\
&+\int_{\Omega\setminus B(x,R_x)} 
\Big|
\Big|\frac{u(x)-e^{\i(x-y)\cdot A(\frac{x+y}{2})} u(y)}{|x-y|}\Big|_p^p-\Big|(\nabla u(x)-\i A(x)u(x))\cdot \frac{x-y}{|x-y|}\Big|_p^p
\Big|
\rho_{m}(x-y) dy\\
&\leq C\int_{B(x,R_x)} |x-y|\rho_m(x-y) dy+ C\int_{\Omega\setminus B(x,R_x)}\rho_m(x-y) dy \\
&\leq C\int_0^{R_x} \rho_m(r)r^N dr+ C\int_{R_x}^{\infty}\rho_m(r)r^{N-1}dr,
\end{align*}
where to handle the second integral we used that
$$
\Big|
\Big|\frac{u(x)-e^{\i(x-y)\cdot A(\frac{x+y}{2})} u(y)}{|x-y|}\Big|_p^p-\Big|(\nabla u(x)-\i A(x)u(x))\cdot \frac{x-y}{|x-y|}\Big|_p^p
\Big|\leq C,\quad \text{for all $x,y\in\Omega$}.
$$
Letting $m\to \infty$ and recalling \eqref{prorho} and Remark~\ref{remconv} we get $\Psi_m(x)\to 0$ for every $x\in\Omega$. Since 
$$
|\Psi_m(x)|\leq C\int_{\Omega} \rho_m(x-y) dy\leq C\int_0^{\infty} \rho_m(r)r^{N-1}dr\leq C,
$$
the Dominated Convergence Theorem yields $\Psi_m\to 0$ in $L^1(\Omega)$ as $m\to \infty$. Then, to get the assertion, it is sufficient to prove that
\begin{align*}
\lim_{m\to \infty}\int_{\Omega}\int_{\Omega}\frac{|(\nabla u(x)-\i A(x)u(x))\cdot (x-y)|_p^p}{|x-y|^p}\rho_{m}(x-y) dydx= pQ_{p,N}\int_{\Omega} |\nabla u-\i Au|_p^pdx.
\end{align*}
Fixed $x\in\Omega$, by virtue of formula \eqref{complex}, we can write
\begin{align*}
 pQ_{p,N}\left|\nabla u(x)-\i A(x)u(x)\right|_p^p&=
\lim_{m\to \infty}\int_{\R^N}\Big|(\nabla u(x)-\i A(x)u(x))\cdot \frac{h}{|h|}\Big|_p^p\rho_{m}(h) dh\\
\nonumber
&=\lim_{m\to \infty}\int_{\Omega}\Big|(\nabla u(x)-\i A(x)u(x))\cdot\frac{x-y}{|x-y|}\Big|_p^p\rho_{m}(x-y) dy\\
\nonumber
&+\lim_{m\to \infty}\int_{\R^N\setminus\Omega}\Big|(\nabla u(x)-\i A(x)u(x))\cdot\frac{x-y}{|x-y|}\Big|_p^p\rho_{m}(x-y) dy.
\end{align*}
To conclude the proof it suffices to prove that
\begin{equation*}
\lim_{m\to \infty}\int_{\Omega}\int_{\R^N\setminus\Omega}\Big|(\nabla u(x)-\i A(x)u(x))\cdot\frac{x-y}{|x-y|}\Big|_p^p\rho_{m}(x-y) dydx=0.
\end{equation*}
For every $\lambda>0$, we denote 
\[
\Omega_{\lambda}:=\{x\in\Omega\ |\ {\rm dist}(x,\partial\Omega)>\lambda\},
\]
and $M:=\| \nabla u-\i Au \|_{L^{\infty}(\Omega)}^p$. 
Then we obtain
\begin{align*}
&\int_{\Omega}\int_{\R^N\setminus\Omega}\Big|(\nabla u(x)-\i A(x)u(x))\cdot\frac{x-y}{|x-y|}\Big|_p^p\rho_{m}(x-y) dydx\\
&=\int_{\Omega}\int_{(\R^N\setminus\Omega)\cap B(x,\lambda)}\Big|(\nabla u(x)-\i A(x)u(x))\cdot\frac{x-y}{|x-y|}\Big|_p^p\rho_{m}(x-y) dydx\\
&+\int_{\Omega}\int_{(\R^N\setminus\Omega)\cap B(x,\lambda)^c}\Big|(\nabla u(x)-\i A(x)u(x))\cdot\frac{x-y}{|x-y|}\Big|_p^p\rho_{m}(x-y) dydx\\
&=\int_{\Omega\setminus\Omega_{\lambda}}\int_{(\R^N\setminus\Omega)\cap B(x,\lambda)}\Big|(\nabla u(x)-\i A(x)u(x))\cdot\frac{x-y}{|x-y|}\Big|_p^p\rho_{m}(x-y) dydx\\
&+\int_{\Omega}\int_{(\R^N\setminus\Omega)\cap B(x,\lambda)^c}\Big|(\nabla u(x)-\i A(x)u(x))\cdot\frac{x-y}{|x-y|}\Big|_p^p\rho_{m}(x-y) dydx\\
&\leq M\int_{\Omega\setminus\Omega_{\lambda}}\int_{(\R^N\setminus\Omega)\cap B(x,\lambda)}\rho_{m}(x-y) dxdy+M\int_{\Omega}\int_{(\R^N\setminus\Omega)\cap B(x,\lambda)^c}\rho_{m}(x-y) dydx\\
&\leq M|\Omega\setminus\Omega_{\lambda}|\int_{\{|h|\leq \lambda\}} \rho_{m}(h) dh+ M|\Omega|\int_{\{|h|> \lambda\}} \rho_{m}(h) dh,
\end{align*}
the assertion follows by letting $m\to\infty$, recalling formula \eqref{prorho},
and finally letting $\lambda\to 0$. If $p=1$ the thesis follows recalling Lemma~\ref{Bv-lemmino}.
\end{proof}

\section{Proof of Theorem \ref{Main2}}\label{main_res2}
We state in the following a few elementary inequalities concerning
the norm introduced in \eqref{p-norm}.

\begin{lemma}\label{propnorm}
	The following properties of $|\cdot|_p$ are true:
	\begin{itemize}
		\item[(i)] Let $m=N$ or $m=1$. There exists a positive constant $C=C(p,N)$ such that 
	$|z \cdot w|_p \leq C \, |z|_p |w|_p,$ for all $z \in \C^m, w \in \C^N$.
		\item[(ii)] If $T:\R^{N}\to \C^N$ is a $C^1$ function, there exists a positive constant $C$ such that
		$$
		\Big|\Big| T(x)\cdot \frac{x-y}{|x-y|}+{\mathcal{O}}(|x-y|)\Big|^p_p-\Big|T(x)\cdot \frac{x-y}{|x-y|}\Big|^p_p\Big|\leq C|x-y|,
		$$
		for all $x,y\in\Omega$, where ${\mathcal O}(|x-y|)$ denotes any continuous function $R:\R^{2N}\to \C$ such that $|R(x,y)|_p|x-y|^{-1}$ is bounded in $\Omega\times \Omega$.
	\end{itemize}
\end{lemma}
\begin{proof}
	To prove $(i)$ we proceed as follows: let $z \in \C^N$,
	\begin{align*}
	|z\cdot w|^p_p=\Big|\sum_{j=1}^N z_jw_j\Big|^p_p&=\Big(\Big|\sum_{j=1}^N \Re z_j\Re w_j -\Im z_j\Im w_j+ \i\Big(\Re z_j\Im w_j+\Im z_j\Re w_j\Big)\Big|_p\Big)^p\\
	&=\Big|\sum_{j=1}^N \Re z_j\Re w_j -\Im z_j\Im w_j\Big|^p+ \Big|\sum_{j=1}^N\Re z_j\Im w_j+\Im z_j\Re w_j\Big|^p\\
	&	\leq C(p)\Big(\Big|\sum_{j=1}^N \Re z_j\Re w_j\Big|^p+\Big|\sum_{j=1}^N \Im z_j\Im w_j\Big|^p+\Big|\sum_{j=1}^N \Re z_j\Im w_j\Big|^p+\Big|\sum_{j=1}^N \Im z_j\Re w_j\Big|^p\Big)\\
	& \leq C(p)\left(| \Re z|^p |\Re w|^p+ |\Im z|^p|\Im w|^p+ |\Re z|^p|\Im w|^p+ |\Im z|^p|\Re w|^p\right)\\
	& =C\, |z|_p^p \, |w|_p^p.
	\end{align*}
	The case $m=1$, i.e. $z \in \C$, works in a similar way.\\
	To prove $(ii)$, it is sufficient to combine the inequality 
	$|b^p-a^p|\leq M(a^{p-1}+b^{p-1})|b-a|$ for
	$$
	a:=\left| T(x)\cdot \frac{x-y}{|x-y|}+{\mathcal{O}}(|x-y|)\right|_p,\qquad
	b:=\left|T(x)\cdot \frac{x-y}{|x-y|}\right|_p,
	$$
	with the triangular inequality
	$$
	\left|\left| T(x)\cdot \frac{x-y}{|x-y|}+{\mathcal{O}}(|x-y|)\right|_p-\left|T(x)\cdot \frac{x-y}{|x-y|}\right|_p\right|\leq |{\mathcal{O}}(|x-y|)|_p\leq C|x-y|,
	$$
	taking into account that $a,b$ are bounded in $\Omega$.
\end{proof}
We start with the following lemma.

\begin{lemma}
	\label{stima1}
	Let $A:\R^N\to\R^N$ be locally bounded. Then, 
	for any	compact $V\subset \R^N$ with $\Omega \Subset V$,
	there exists $C=C(A,V)>0$ such that
	\begin{equation*}
	\int_{\R^{n}} |u(y+h)-e^{\i  h\cdot A\left(y+\frac{h}{2}\right)}u(y)|^{p}_pdy\leq C |h|^{p} \|u\|^{p}_{W^{1,p}_{A}(\R^{n})}, 
	\end{equation*}
	for all $u\in W^{1,p}_{A}(\R^{N})$ such that $u=0$ on $V^c$ and any $h\in \R^{N}$ with $|h|\leq 1$.
\end{lemma}
\begin{proof}
	Assume first that $u\in C_{0}^{\infty}(\R^{N})$ 
	with $u=0$ on $V^c$.  Fix $y, h\in \R^{N}$ and define
	\[
	\varphi(t):=e^{\i (1-t) h\cdot A\left(y+\frac{h}{2}\right)}u(y+th),\quad\, t\in [0,1].
	\]
	Then we have
	$
	u(y+h)-e^{\i  h\cdot A\left(y+\frac{h}{2}\right)}u(y)=\int_{0}^{1}\varphi^{\prime}(t)dt,
	$
	and since
	\[
	\varphi^{\prime}(t)=e^{\i (1-t) h\cdot A\big(y+\frac{h}{2}\big)}\,h\cdot\Big(\nabla_{y}u(y+th)-\i A\Big(y+\frac{h}{2}\Big)u(y+th)\Big),
	\]
	by H\"older inequality and recalling that $|e^{\i (1-t) h\cdot A\big(y+\frac{h}{2}\big)}|_p\leq 2$ we get
	\[
	|u(y+h)-e^{\i  h\cdot A\left(y+\frac{h}{2}\right)}u(y)|^{p}_p\leq 2|h|^{p}\int_{0}^{1} \Big|\nabla_{y}u(y+th)-\i A\Big(y+\frac{h}{2}\Big)u(y+th)\Big|^{p}_p dt.
	\]
	Therefore, integrating with respect to $y$ over $\R^N$ and using Fubini's Theorem, we get
	\begin{align*}
	\int_{\R^{N}} |u(y+h)-e^{\i  h\cdot A\left(y+\frac{h}{2}\right)}u(y)|^{p}_pdy &\leq 2|h|^{p}
	\int_{0}^{1}dt\int_{\R^{n}} \Big|\nabla_{y}u(y+th)-\i A\Big(y+\frac{h}{2}\Big)u(y+th)\Big|^{p}_p dy\nonumber\\
	&=2|h|^{p} \int_{0}^{1}dt\int_{\R^{N}} \Big|\nabla_{z}u(z)-\i A\Big(z+\frac{1-2t}{2}h\Big)u(z)\Big|^{p}_pdz\nonumber\\
	&\leq C |h|^{p} \int_{\R^{n}} |\nabla_{z}u(z)-\i A\left(z\right)u(z)|^{p}_pdz \\&
	+C|h|^{p} \int_{V} \Big|A\Big(z+\frac{1-2t}{2}h\Big)-A(z)\Big|^{p}_p|u(z)|^{p}_pdz.
	\end{align*}
	Then, since $A$ is bounded on the set $V$, we have for some constant $C>0$
	\begin{align*}
	\int_{\R^{N}} |u(y+h)-e^{\i  h\cdot A\left(y+\frac{h}{2}\right)}u(y)|^{p}_p dy&\leq C |h|^{p}\left( \int_{\R^{N}} |\nabla_{z}u(z)-\i
	A\left(z\right)u(z)|^{p}_pdz+ \int_{\R^{n}} |u(z)|^{p}_p dz\right)\nonumber\\&
	=C |h|^{p} 
	\|u\|^{p}_{W^{1,p}_{A}(\R^{N})}.
	\end{align*}
	When dealing with a general $u$ we can argue by a density argument \cite[Theorem 7.22]{LL}.
\end{proof}

\begin{lemma}\label{firstLemma}
	Let $A:\R^N\to\R^N$ be locally bounded.
	Let $u\in W^{1,p}_{A}(\Omega)$ and $\rho\in L^{1}(\R^{N})$ with $\rho\geq0$. Then
	\[
	\int_{\Omega}\int_{\Omega}\frac{|u(x)-e^{\i (x-y)\cdot A\left(\frac{x+y}{2}\right)}u(y)|^p_p}{|x-y|^{p}}\rho(x-y)\,dxdy\leq C\|\rho\|_{L^{1}}\|u\|^{p}_{W^{1,p}_{A}(\Omega)}
	\]
	where $C$ depends only on $\Omega$ and $A$.
\end{lemma}
\begin{proof}
	Let $V\subset \R^N$ be a fixed compact set with $\Omega\Subset V$. 
	Given $u\in W^{1,p}_A(\Omega)$, 
	there exists 
	$\tilde u\in W^{1,p}_{A}(\R^{N})$ with $\tilde u=u$ on $\Omega$ and $\tilde u=0$ on $V^c$ (see e.g. \cite[Lemma 2.2]{BM}).
	By Lemma \ref{stima1}, we obtain
	\begin{equation}
	\int_{\R^{N}} |\tilde u(y+h)-e^{\i  h\cdot A\left(y+\frac{h}{2}\right)}\tilde u(y)|^{p}_p dy
	\leq C |h|^{p} \|\tilde u\|^{p}_{W^{1,p}_{A}(\R^{N})}\leq C |h|^{p} \|u\|^{p}_{W^{1,p}_{A}(\Omega)}, \label{secondeq}
	\end{equation}
	for some positive constant $C$ depending on $\Omega$ and $A$. Then, in light of \eqref{secondeq}, we get
	\begin{align*}
	\int_{\Omega}\int_{\Omega}\frac{|u(x)-e^{\i (x-y)\cdot A\left(\frac{x+y}{2}\right)}u(y)|^p_p}{|x-y|^{p}}\rho(x-y)\,dxdy
	&\leq\int_{\R^{N}}\int_{\R^{N}}\rho(h)\frac{|\tilde u(y+h)-e^{\i  h\cdot A\left(y+\frac{h}{2}\right)}\tilde u(y)|^{p}_p}{|h|^{p}}dydh\nonumber
	\\& =\int_{\R^{N}}\frac{\rho(h)}{|h|^{p}}\Big(\int_{\R^{N}}|\tilde u(y+h)-e^{\i  h\cdot A\left(y+\frac{h}{2}\right)}\tilde u(y)|^{p}_p dy\Big)dh
	\\& \leq C\|\rho\|_{L^{1}} \|u\|^{p}_{W^{1,p}_{A}(\Omega)},
	\end{align*}
	concluding the proof.
\end{proof}

\noindent
We can now conclude the proof of Theorem~\ref{Main2}. Setting
\[
F_{m}^{u}(x,y):=\frac{u(x)-e^{\i (x-y)\cdot A\left(\frac{x+y}{2}\right)}u(y)}{|x-y|}\rho_{m}^{1/p}(x-y),
\quad\,\,\, x,y\in\Omega,\,\, m\in\N,
\]
by virtue of Lemma \ref{firstLemma}, for all $u,v\in W^{1,p}_{A}(\Omega)$, 
we have (recall that $\rho_m$ fulfills condition \eqref{lim-cond})
\[
\big|\|F_{m}^{u}\|_{L^p(\Omega\times\Omega)}-\|F_{m}^{v}\|_{L^p(\Omega\times\Omega)}\big|\leq\|F_{m}^{u}-F_{m}^{v}\|_{L^p(\Omega\times\Omega)}
\leq C  \|u-v\|_{W^{1,p}_{A}(\Omega)},
\]
for some $C>0$ depending on $\Omega$ and $A$.
This allows to prove the assertion for functions $u\in C^{2}(\bar{\Omega})$
since for every $u\in W^{1,p}_A(\Omega)$ there is a sequence $\{u_j\}_{j\in\N}
\subset C^\infty(\overline{\Omega})$ such that $\|u_j-u\|_{W^{1,p}_{A}(\Omega)}\to 0$.
Therefore, the assertion follows by Proposition~\ref{form-smooth}.

\section{Proof of Theorem \ref{Main}}
\label{sec6}

\noindent
We first state a technical lemma.

\begin{lemma}\label{proppsi}
Let $\Omega\subset \R^N$ be open and bounded and $A \in C^{2}(\R^{N},\R^N)$ and 
$R>0$. For $x,y \in \Omega$ let 
$$
\psi(z):=e^{\i(x-y)\cdot A\left(\frac{x+y}{2}+z\right)}, \quad z \in B(0,R).
$$
Then there exist positive constants $D_1=D_1(A, \Omega)$ and $D_2=D_2(A, \Omega, R)$ such that 
\begin{align}\label{ineqpsi}
\left|\psi(z)-\psi(0)\right|_1\leq &D_1\left|z\right|\left|x-y\right|+D_2|z|^2\left|x-y\right|,
\end{align}
for every $z\in B(0,R)$. Moreover, $\limsup_{R\to 0} D_2<\infty$.
\end{lemma}
\begin{proof}
Recalling \eqref{p-norm}, we can prove \eqref{ineqpsi} separately
for the real part $\Re \psi$ and the imaginary part $\Im \psi$.
To simplify the notation, for fixed $x,y \in \Omega$, let us denote 
$$
\theta(z) := (x-y) \cdot A\left( \dfrac{x+y}{2}+z\right), \quad  z \in B(0,R).
$$
Therefore, 
$$\psi(z) = \Re \psi(z) + \i \, \Im \psi(z) = \cos(\theta(z)) + \i \, \sin(\theta(z)), \quad z\in B(0,R).$$
We start considering first the real part $\Re \psi$. By Taylor's formula with Lagrange's rest,
we have
\begin{align}\label{taylor}
\Re \psi(z)- \Re \psi(0)=\nabla \Re \psi(0) \cdot z +\frac{1}{2} \nabla^2 \Re\psi(\bar tz) z\cdot z,
\end{align}
for some $\bar t\in [0,1]$,
where $\nabla^{2} \Re\psi$ stands for the Hessian matrix of $\Re \psi$.
A simple computation gives
\begin{equation*}
\partial_{z_{j}} \Re \psi (z) = - \sin(\theta(z)) \, \partial_{z_{j}}\theta(z) = 
-\sin(\theta(z))\, \sum_{k=1}^{N}(x_{k}-y_{k}) \partial_{z_{j}}A^{(k)}\left( \dfrac{x+y}{2} + z\right),
\end{equation*}
for every $j=1,\ldots,N$. Therefore, we have
\begin{align}\label{grad}
\nabla \Re \psi(0)=-\sin\Big( (x-y)\cdot A\Big(\dfrac{x+y}{2} \Big) \Big) \, (x-y) \nabla A\left(\dfrac{x+y}{2}\right),
\end{align}
where $\nabla A$ denotes the Jacobian matrix of $A$.
Another quite simple computation yields
\begin{equation}\label{hess}
\begin{aligned}
\left(\nabla^2 \Re \psi(z)\right)_{h,j} &
= - \Bigg[ \cos(\theta(z))\left((x-y)\cdot \partial_{z_h}A\Big(\frac{x+y}{2}+z\Big)\right)\left((x-y)\cdot \partial_{z_j}A\Big(\frac{x+y}{2}+z\Big)\right)\\
&\quad + \sin(\theta(z)) (x-y)\cdot \partial_{z_h}\partial_{z_j}A \Big(\frac{x+y}{2}+z\Big) \Bigg], 
\end{aligned}\end{equation}
for every $i,j=1,\ldots, N$.
Now, using \eqref{taylor} and \eqref{grad} we get
\begin{align}
\left|\Re \psi(z)-\Re \psi(0)\right|\leq \left|\nabla A\left(\frac{x+y}{2}\right)\right||z||x-y|+ \frac{1}{2} |z|^2 |\nabla^2 \Re \psi(\overline{t}z)|, \qquad \mbox{for some}\ \overline{t}\in [0,1].
\end{align}
On the other hand, by \eqref{hess} we get
\begin{equation*}
|\nabla^2\Re \psi(\overline{t}z)|
\leq |x-y|\left(C |x-y| \left|\nabla A\left(\frac{x+y}{2}+\overline{t}z\right)\right|^{2}+ \sum_{k=1}^N\left|\nabla^2 A^{(k)}\left(\frac{x+y}{2}+\overline{t}z\right)\right|\right).
\end{equation*}

\noindent Therefore, \eqref{ineqpsi} for $\Re \psi$ follows taking 
\[
D_1:=\sup_{x,y\in \Omega} \Big|\nabla A\Big(\frac{x+y}{2}\Big)\Big|<\infty
\]
and
\[D_2:=\frac{1}{2}\sup_{\overset{x,y\in\Omega}{z\in B(0,R)}}\sum_{k=1}^N\Big|\nabla^2 A^{(k)}\left(\frac{x+y}{2}+\overline{t}z\right)\Big|+ C |x-y|\Big|\nabla A\left(\frac{x+y}{2}+\overline{t}z\right)\Big|^{2}<\infty.\]
The fact that $\limsup_{R \to 0} D_2<\infty$ follows observing that $D_2$ decreases as $R$ decreases.
Since a similar argument holds for $\Im \psi$, we get the assertion.
\end{proof}

\begin{lemma}
\label{approx}
Let $\Omega \subset \R^N$ be an open set and $A\in C^2(\R^N,\R^N)$.
Let $u\in L^1(\Omega)$. Denote by $u_\eps$ its regularization as 
defined in \eqref{mollif}. 
Define
\[
\Omega_r:=\{x\in\Omega\ |\ d(x,\partial\Omega)>r\}, \quad \forall r>0.
\]
Then, for all $r>0$ and $\eps\in (0,r)$ there holds
\begin{align*}
&\int_{\Omega_r}\int_{\Omega_r}\frac{|u_{\varepsilon}(x)-e^{\i(x-y)\cdot A(\frac{x+y}{2})}u_{\varepsilon}(y)|_1}{|x-y|}\rho_m(x-y) dxdy\\
\nonumber
&\leq \int_{\Omega}\int_{\Omega}\frac{|u(x)-e^{\i(x-y)\cdot A(\frac{x+y}{2})}u(y)|_1}{|x-y|}\rho_m(x-y) dxdy\\
\nonumber
&+\frac{1}{\varepsilon^N}\int_{B(0,\varepsilon)}\eta\left(\frac{z}{\varepsilon}\right)\int_{\Omega}\int_{\Omega} \frac{\left|e^{\i(x-y)\cdot A(\frac{x+y}{2}+z)}u(y)-e^{\i(x-y)\cdot A(\frac{x+y}{2})}u(y)\right|_1}{|x-y|} \rho_m(x-y) dxdydz.
\end{align*}
and
\begin{equation*}
\lim_{\varepsilon\to 0}\lim_{m\to\infty} \frac{1}{\varepsilon^N}\int_{B(0,\varepsilon)}\eta\left(\frac{z}{\varepsilon}\right)\int_{\Omega}\int_{\Omega} \frac{\left|e^{\i(x-y)\cdot A(\frac{x+y}{2}+z)}u(y)-e^{\i(x-y)\cdot A(\frac{x+y}{2})}u(y)\right|_1}{|x-y|} \rho_m(x-y) dxdydz=0.
\end{equation*}
\end{lemma}
\begin{proof}
Let us extend $u$ to the whole of $\R^N$ by zero. 
To simplify the notation, let us still denote by $u$ its extension. 
By definition,
\begin{align*}
u_{\varepsilon}(x)-e^{\i(x-y)\cdot A(\frac{x+y}{2})}u_{\varepsilon}(y)&=\frac{1}{\varepsilon^N}\int_{\R^N}\eta\left(\frac{z}{\varepsilon}\right) (u(x-z)-e^{\i(x-y)\cdot A(\frac{x+y}{2})}u(y-z)) dz\\
\nonumber
&=\frac{1}{\varepsilon^N}\int_{B(0,\varepsilon)}\eta\left(\frac{z}{\varepsilon}\right) (u(x-z)-e^{\i(x-y)\cdot A(\frac{x+y}{2})}u(y-z)) dz.
\end{align*}
Thus, for every $\eps\in (0,r)$, there holds
\begin{align*}
&\int_{\Omega_r}\int_{\Omega_r}\frac{\left|u_{\varepsilon}(x)-e^{\i(x-y)\cdot A(\frac{x+y}{2})}u_{\varepsilon}(y)\right|_1}{|x-y|}\rho_m(x-y) dxdy\\
\nonumber
&\leq \frac{1}{\varepsilon^N}\int_{\Omega_r}\int_{\Omega_r}\int_{B(0,\varepsilon)}\eta\left(\frac{z}{\varepsilon}\right) \frac{\left|u(x-z)-e^{\i(x-y)\cdot A(\frac{x+y}{2})}u(y-z)\right|_1}{|x-y|} \rho_m(x-y) dzdxdy\\
\nonumber
&\leq \frac{1}{\varepsilon^N}\int_{B(0,\varepsilon)}\eta\left(\frac{z}{\varepsilon}\right)\int_{\Omega}\int_{\Omega} \frac{\left|u(x)-e^{\i(x-y)\cdot A(\frac{x+y}{2}+z)}u(y)\right|_1}{|x-y|} \rho_m(x-y) dxdydz
\leq \mathcal{I} + \mathcal{II},
\end{align*}
where
\begin{align*}
\mathcal{I}:&= \frac{1}{\varepsilon^N}\int_{B(0,\varepsilon)}\eta\left(\frac{z}{\varepsilon}\right)\int_{\Omega}\int_{\Omega} \frac{\left|u(x)-e^{\i(x-y)\cdot A(\frac{x+y}{2})}u(y)\right|_1}{|x-y|} \rho_m(x-y) dxdydz\\
\nonumber
&=\int_{\Omega}\int_{\Omega} \frac{\left|u(x)-e^{\i(x-y)\cdot A(\frac{x+y}{2})}u(y)\right|_1}{|x-y|} \rho_m(x-y) dxdy
\end{align*}
and
\[
\mathcal{II}:=\frac{1}{\varepsilon^N}\int_{B(0,\varepsilon)}\eta\left(\frac{z}{\varepsilon}\right)\int_{\Omega}\int_{\Omega} \frac{\left|e^{\i(x-y)\cdot A(\frac{x+y}{2}+z)}u(y)-e^{\i(x-y)\cdot A(\frac{x+y}{2})}u(y)\right|_1}{|x-y|} \rho_m(x-y) dxdydz.
\]
Define $\psi(z):=e^{\i(x-y)\cdot A\left(\frac{x+y}{2}+z\right)}$. Then
$|\psi(z)|_1\leq 2$ for all $z\in B(0,\varepsilon)$ 
and by Lemma \ref{proppsi}
\begin{align*}
\left|\psi(z)-\psi(0)\right|_1\leq D_1\left|z\right|\left|x-y\right|+ D_2|z|^2\left|x-y\right|\qquad \forall x,y\in \Omega, z\in B(0,\varepsilon),
\end{align*}
for some $D_1=D_1(A,\Omega)$ and $D_2=D_2(A,\Omega,\eps)$
which is bounded as $\eps\searrow 0$. Therefore,
\begin{align*}
&\mathcal{II}\leq \frac{D_1}{\varepsilon^N}\int_{B(0,\varepsilon)}\eta\left(\frac{z}{\varepsilon}\right)\int_{\Omega}\int_{\Omega} \left|u(y)\right|_1\left|z\right|\rho_m(x-y) dxdydz+\\
\nonumber
&+\frac{D_2}{\varepsilon^N}\int_{B(0,\varepsilon)}\eta\left(\frac{z}{\varepsilon}\right)\int_{\Omega}\int_{\Omega}\left|u(y)\right|_1|z|^2\rho_m(x-y) dxdydz.
\end{align*}
We have 
\begin{align*}
&\frac{D_2}{\varepsilon^N}\int_{B(0,\varepsilon)}\eta\left(\frac{z}{\varepsilon}\right)\int_{\Omega}\int_{\Omega}\left|u(y)\right|_1|z|^2 \rho_m(x-y) dxdydz\\
\nonumber
&\leq \frac{D_2}{\varepsilon^N}\int_{B(0,\varepsilon)}\eta\left(\frac{z}{\varepsilon}\right)|z|^2 dz \int_{\Omega}\left|u(y)\right|_1\left(\int_{\Omega}\rho_m(x-y)dx\right)dy\leq 
2D_2|\mathbb{S}^{N-1}| \|u\|_{L^1(\Omega)}\varepsilon^2,
\end{align*}
since $\int_{\Omega}\rho_m(x-y) dx\leq |\mathbb{S}^{N-1}|\int_0^\infty \rho_m(r)r^{N-1}dr\leq 2|\mathbb{S}^{N-1}|$, in view of
\eqref{lim-cond}. Analogously, we have
$$
\frac{D_1}{\varepsilon^N}\int_{B(0,\varepsilon)}\eta\left(\frac{z}{\varepsilon}\right)\int_{\Omega}\int_{\Omega} \left|u(y)\right|_1\left|z\right|\rho_m(x-y) dxdydz
\leq  2D_1|\mathbb{S}^{N-1}|\|u\|_{L^1(\Omega)}\varepsilon,
$$
Hence, we conclude that
\begin{align*}
\lim_{\varepsilon\to 0}\lim_{m\to\infty}\mathcal{II}=0,
\end{align*}
and the thesis follows.
\end{proof}

\begin{lemma}\label{stima}
Let $\Omega \subset \R^N$ be an open and bounded set.
Denote by $xy_t:=tx+(1-t)y$
with $t\in [0,1]$ the linear combination of $x,y \in \Omega$. 
There exists a positive constant  
$C=C(N,\Omega,A)$ such that  
\begin{align*}
&\int_{\Omega}\int_{\Omega}\int_0^1 \left|\Big(e^{\i(1-t)(x-y)\cdot A \left(\frac{x+y}{2}\right)}-1\Big)\frac{x-y}{|x-y|}\cdot\left( \nabla_y u(xy_t)-\i A\left(\frac{x+y}{2}\right) u(xy_t)\right)\right|_1 \rho_m(x-y) dtdxdy \\
& \leq C \|u\|_{BV_A(W)}\Big(\int_0^1 r^N\rho_m(r)dr+\int_1^\infty r^{N-1}\rho_m(r)dr\Big),
\end{align*}
for every open set $W\Supset\Omega$ and for every $u\in C^2(\R^N,\C)$ such that $u=0$ on $W^c$.
\begin{proof}
It is readily seen that there exists a positive constant $C=C(A,\Omega)$ such that
	\begin{equation}\label{Estim}
	\left|e^{\i(1-t)(x-y)\cdot A \left(\frac{x+y}{2}\right)}-1\right|_1 \leq C|x-y|,
	\quad\,\, \text{for all $x,y\in\Omega$ and all $t\in [0,1]$}.
	\end{equation}
	Then, by $(i)$ of Lemma \ref{propnorm} with $p=1$ and by
	\eqref{Estim}, we have
	\begin{align*}
	&\int_{\Omega}\int_{\Omega}\int_0^1 \left|\Big(e^{\i(1-t)(x-y)\cdot A \left(\frac{x+y}{2}\right)}-1\Big)\frac{x-y}{|x-y|}\cdot\left( \nabla_y u(xy_t)-\i A\left(\frac{x+y}{2}\right) u(xy_t)\right)\right|_1 \rho_m(x-y) dtdxdy  \\
	& \leq C\int_{\Omega}\int_{\Omega}\int_0^1 \left|e^{\i(1-t)(x-y)\cdot A \left(\frac{x+y}{2}\right)}-1\right|_1 \Big|\nabla_y u(xy_t)-\i A\left(\frac{x+y}{2}\right) u(xy_t)\Big|_1\rho_m(x-y) dtdxdy  \\
	& \leq C\int_{\Omega}\int_{\Omega}\int_0^1 |x-y|\rho_m(x-y)\Big|\nabla_y u(xy_t)-\i A\left(\frac{x+y}{2}\right) u(xy_t)\Big|_1 dtdxdy
	\leq \mathcal{I}+\mathcal{II}
	\end{align*}
	where we have set
	\begin{align*}
	& \mathcal{I}:=C\int_{\Omega}\int_{\Omega}\int_0^1 |x-y|\rho_m(x-y) \Big|\nabla_y u(xy_t)-\i A\left(xy_t\right) u(xy_t)\Big|_1 dtdxdy,\\
	& \mathcal{II}:=C\int_{\Omega}\int_{\Omega}\int_0^1 |x-y| \rho_m(x-y)\Big| A\left(\frac{x+y}{2}\right)- A(xy_t)\Big| \left|u(xy_t)\right|_1 dtdxdy  
	\end{align*}
	for some positive constant 
	$C=C(A,\Omega)$. Then we get 
	\begin{equation*}
	\begin{aligned}
	\mathcal{I} &\leq C\int_{\Omega}\left(\int_{B(y,1)\cap \Omega} |x-y|\rho_m(x-y)\left(\int_0^1\Big|\nabla_y u(xy_t)-\i A\left(xy_t\right) u(xy_t)\Big|_1dt\right)dx\right)dy \\
	&+C\int_{\Omega}\left(\int_{B(y,1)^c \cap \Omega}\rho_m(x-y)\left(\int_0^1\Big|\nabla_y u(xy_t)-\i A\left(xy_t\right) u(xy_t)\Big|_1dt\right)dx\right)dy \\
	&\leq C \int_{\R^N}\left(\int_{B(0,1)} |z|\rho_m(z)\left(\int_0^1\Big|\nabla_y u(y+tz)-\i A\left(y+tz\right) u(y+tz)\Big|_1 dt\right)dz\right)dy \\
	&+C \int_{\R^N}\left(\int_{B(0,1)^c} \rho_m(z)\left(\int_0^1\Big|\nabla_y u(y+tz)-\i A\left(y+tz\right) u(y+tz)\Big|_1 dt\right)dz\right)dy \\
	&\leq C \int_{B(0,1)} |z|\rho_m(z)\left(\int_{\R^N}\int_0^1\Big|\nabla_y u(y+tz)-\i A\left(y+tz\right) u(y+tz)\Big|_1 dtdy\right)dz \\
	&+C \int_{B(0,1)^c} \rho_m(z)\left(\int_{\R^N}\int_0^1 \Big|\nabla_y u(y+tz)-\i A\left(y+tz\right) u(y+tz)\Big|_1 dtdy\right)dz	\\
	& \leq C\left(\int_{W}\Big|\nabla_y u(z)-\i A\left(z\right) u(z)\Big|_1dz\right)\Big(\int_0^1 r^N\rho_m(r)dr+\int_1^\infty r^{N-1}\rho_m(r)dr\Big),
	\end{aligned}\end{equation*}
	where in the last inequality we used
	$$
	\int_{\R^N}\int_0^1 \Big|\nabla_y u(y+tz)-\i A\left(y+tz\right) u(y+tz)\Big|_1 dtdy=\int_{\R^N}\Big|\nabla_y u(z)-\i A\left(z\right) u(z)\Big|_1dz
	$$
	as well as
	$$
	\int_{W}\Big|\nabla_y u(z)-\i A\left(z\right) u(z)\Big|_1dz=\int_{\R^N}\Big|\nabla_y u(z)-\i A\left(z\right) u(z)\Big|_1dz.
	$$
On the other hand, denoting by ${\rm Conv}(\Omega)$ the convex hull of $\Omega$,
and arguing in a similar fashion, 
one obtains
\begin{align*}
\mathcal{II} &\leq C\|A\|_{L^\infty({\rm Conv}(\Omega))}
\left(\int_{W}\Big|u(z)\Big|_1dz\right)\Big(\int_0^1 r^N\rho_m(r)dr+\int_1^\infty r^{N-1}\rho_m(r)dr\Big) \\
& \leq C\left(\int_{W}\Big|u(z)\Big|_1dz\right)\Big(\int_0^1 r^N\rho_m(r)dr+\int_1^\infty r^{N-1}\rho_m(r)dr\Big),
\end{align*}
for some positive constant $C=C(A,\Omega)$.\
The desired assertion finally follows by combining the above inequalities and then using Lemma~\ref{Bv-lemmino}.
\end{proof}	
\end{lemma}

\noindent
The following lemma is an adaptation to our case of \cite[Lemma 3]{davila} and \cite[Lemma 5.2]{Ponce}. 

\begin{lemma}\label{Ponciarello}
Let $A:\R^N\to\R^N$ be locally Lipschitz and $\Omega\subset\R^N$ 
be an open and bounded set. Then there exists a positive constant $C=C(\Omega,A)$ 
such that for all $r,m>0$, $W\Supset\Omega$ (i.e. $\Omega$ is compactly contained in $W$) and $u\in BV_A(\Omega)$, 
denoting by $\overline{u}\in BV_A(\R^N)$ an extension of $u$ to $\R^N$ 
such that $\overline{u}=0$ in $W^c$, the following inequality holds
\begin{align*}
\int_{\Omega}\int_{\Omega}&\frac{|u(x)-e^{\i(x-y)\cdot A(\frac{x+y}{2})}u(y)|_1}{|x-y|}\rho_m(x-y) dxdy\\
&\leq Q_{1,N}|D\overline{u}|_A(E_r')\int_{0}^r \rho_m(s)s^{N-1} ds+\frac{{\rm Lip}(A,E_r') \|\overline{u}\|_{L^1(W)}}{2}\int_{0}^r s^N\rho_m(s)ds\\
&+C\|\overline{u}\|_{BV_A(W')}\Big(\int_0^1 s^N\rho_m(s)ds+\int_1^\infty s^{N-1}\rho_m(s)ds\Big)\\
&+\frac{C\|\overline{u}\|_{L^1(W)}}{r}\int_{r}^{\infty} s^{N-1} \rho_m(s) ds,
\end{align*}
where $E_r:=\Omega+B(0,r)$, $W'$ (resp.\ $E_r'$) is any bounded open set with $W'\Supset W$ (resp.\ $E_r'\Supset E_r$).
\end{lemma}
\begin{proof} 
For any $\varepsilon\in (0,r),$ let $\overline{u}_\eps$ be as in 
formula \eqref{mollif} for $\overline{u}:\R^N\to\C$. 
By a change of variables, Fubini's Theorem and Lemma \ref{propnorm}, we have
\begin{align*}
\int_{\Omega}&\int_{\Omega}\frac{|\overline{u}_{\varepsilon}(x)-e^{\i(x-y)\cdot A(\frac{x+y}{2})}\overline{u}_{\varepsilon}(y)|_1}{|x-y|}\rho_m(x-y) dxdy\\
\nonumber
&\leq \int_{\Omega}\Big(\int_{\Omega\cap B(y,r)}\frac{|\overline{u}_{\varepsilon}(x)-e^{\i(x-y)\cdot A(\frac{x+y}{2})}\overline{u}_{\varepsilon}(y)|_1}{|x-y|}\rho_m(x-y) dx\Big)dy+\frac{C\|\overline{u}\|_{L^1(W)}}{r}\int_{B(0,r)^c} \rho_m(h) dh,
\end{align*}
where $C=C(N)>0$.
Let us now define
$\psi(t):=e^{\i(1-t)(x-y)\cdot A\left(\frac{x+y}{2}\right)}\overline{u}_{\varepsilon}(tx+(1-t)y), $ $t\in [0,1]$.
Then 
$$
\overline{u}_{\varepsilon}(x)-e^{\i(x-y)\cdot A\left(\frac{x+y}{2}\right)}\overline{u}_{\varepsilon}(y)=\psi(1)-\psi(0)=\int_0^1 \psi'(t) dt,
$$
and since
\[
\psi'(t)=e^{\i(1-t)(x-y)\cdot A\left(\frac{x+y}{2}\right)}(x-y)\cdot\Big( \nabla_y\overline{u}_{\varepsilon}(tx+(1-t)y)-\i A\Big(\frac{x+y}{2}\Big)\overline{u}_{\varepsilon}(tx+(1-t)y)\Big),
\]
we have
\begin{align*}
\int_{\Omega}\Big(\int_{\Omega\cap B(y,r)}\frac{|\overline{u}_{\varepsilon}(x)-e^{\i(x-y)\cdot A(\frac{x+y}{2})}\overline{u}_{\varepsilon}(y)|_1}{|x-y|}\rho_m(x-y) dx\Big)dy\leq \mathcal{I}+\mathcal{II},
\end{align*}
where we have set
\begin{align*}
&\mathcal{I}:= \int_{\Omega}\Big(\int_{\Omega\cap B(y,r)}\int_0^1 \left|\frac{x-y}{|x-y|}\cdot\left( \nabla_y \overline{u}_{\varepsilon}(xy_t)-\i A\left(\frac{x+y}{2}\right) \overline{u}_{\varepsilon}(xy_t)\right)\right|_1 \rho_m(x-y) dtdx\Big)dy \\
&\mathcal{II}:= \Big|\int_{\Omega}\int_{\Omega\cap B(y,r)}\int_0^1 \left|e^{\i(1-t)(x-y)\cdot A \left(\frac{x+y}{2}\right)}\frac{x-y}{|x-y|}\cdot\left( \nabla_y \overline{u}_{\varepsilon}(xy_t)-\i A\left(\frac{x+y}{2}\right) \overline{u}_{\varepsilon}(xy_t)\right)\right|_1 \rho_m(x-y)dtdxdy  \\
 &-\int_{\Omega}\int_{\Omega\cap B(y,r)}\int_0^1 \left|\frac{x-y}{|x-y|}\cdot\left( \nabla_y \overline{u}_{\varepsilon}(xy_t)-\i A\left(\frac{x+y}{2}\right) \overline{u}_{\varepsilon}(xy_t)\right)\right|_1\rho_m(x-y) dtdxdy\Big|.
\end{align*}
Let $W_\eps:=\{x\in\R^N: d(x,W)<\eps\}$, we 
have $\overline{u}_\eps=0$ on $W^c_\eps$ and by Lemmas \ref{stima} and \ref{conv}
\begin{equation}\label{stimaII}
\begin{aligned}
 \mathcal{II} &\leq C \|\overline{u}_{\varepsilon}\|_{BV_A(W_\eps)} \left(\int_0^1 r^N\rho_m(r)dr+\int_1^\infty r^{N-1}\rho_m(r)dr\right)\\
&\leq C \left(\|\overline{u}\|_{BV_A(W')}+\varepsilon \mathrm{Lip}(A,W')\|\overline{u}\|_{L^1(W')}\right)\left(\int_0^1 r^N\rho_m(r)dr+\int_1^\infty r^{N-1}\rho_m(r)dr\right),
\end{aligned}\end{equation}
for an arbitrary open set $W'\Supset W$ and for some positive constant $C=C(N,\Omega, A)$. On the other hand, we have
\begin{align*}
&\mathcal{I} \leq 
\int_{B(0,r)}\int_0^1 \int_{\Omega}\left|\left( \nabla_y\overline{u}_{\varepsilon}(y+th)-\i A\left(y+\frac{h}{2}\right)\overline{u}_{\varepsilon}(y+th)\right)\cdot \frac{h}{|h|}\right|_1 \rho_m(h) dydtdh\\
\nonumber
&\leq \int_{B(0,r)}\int_0^1 \int_{\Omega}\left|\left( \nabla_y\overline{u}_{\varepsilon}(y+th)-\i A\left(y+th\right)\overline{u}_{\varepsilon}(y+th)\right)\cdot \frac{h}{|h|}\right|_1 \rho_m(h) dydtdh\\
\nonumber
&+\int_{B(0,r)}\int_0^1 \int_{\Omega}\left|\left( \i A\left(y+\frac{h}{2}\right)\overline{u}_{\varepsilon}(y+th)-\i A\left(y+th\right)\overline{u}_{\varepsilon}(y+th)\right)\cdot \frac{h}{|h|}\right|_1 \rho_m(h) dydtdh\\
\nonumber
&\leq \int_{B(0,r)} \int_{E_r}\left|\left( \nabla_y\overline{u}_{\varepsilon}(z)-\i A\left(z\right)\overline{u}_{\varepsilon}(z)\right)\cdot \frac{h}{|h|}\right|_1 \rho_m(h) dzdh\\
\nonumber
&+\int_{B(0,r)}\int_0^1 \int_{\Omega}\left|\left( A\left(y+\frac{h}{2}\right)-A\left(y+th\right)\right)\cdot \frac{h}{|h|}\right|_1\left|\overline{u}_{\varepsilon}(y+th)\right|_1 \rho_m(h) dydtdh\\
\nonumber
&\leq \int_{0}^r \int_{E_r}\left(\int_{S^{N-1}}\left|\left( \nabla_y\overline{u}_{\varepsilon}(z)-\i A\left(z\right)\overline{u}_{\varepsilon}(z)\right)\cdot w\right|_1   d\mathcal{H}^{N-1}(w)\right)s^{N-1}\rho_m(s)dzds\\
\nonumber
&+\int_{B(0,r)}\int_0^1 \int_{\Omega}\left|\left( A\left(y+\frac{h}{2}\right)-A\left(y+th\right)\right)\cdot \frac{h}{|h|}\right|_1\left|\overline{u}_{\varepsilon}(y+th)\right|_1 \rho_m(h) dydtdh.
\end{align*}
Taking into account that (see the final lines of the proof of Lemma~\ref{rem})
$$
\int_{S^{N-1}}\left|\xi\cdot w\right|_1 d\mathcal{H}^{N-1}(w)=Q_{1,N}|\xi|_1,
\quad\text{for any $\xi\in \C^N$},
$$
we obtain that 
\begin{align*}
& \mathcal{I} \leq Q_{1,N}\int_{0}^r \int_{E_r}\left|\nabla_y\overline{u}_{\varepsilon}(z)-\i A\left(z\right)\overline{u}_{\varepsilon}(z)\right|_1  s^{N-1}\rho_m(s) dsdz\\
\nonumber
&+\int_{B(0,r)}\int_0^1 \int_{\Omega}\left|\left( A\left(y+\frac{h}{2}\right)-A\left(y+th\right)\right)\cdot \frac{h}{|h|}\right|_1\left|\overline{u}_{\varepsilon}(y+th)\right|_1 \rho_m(h) dydtdh.
\end{align*}
Whence, taking into account Lemma~\ref{NormaSmooth} and Lemma~\ref{conv}, we finally get 
\begin{align}\label{stimaI}
& \mathcal{I} \leq Q_{1,N}\left(\int_{E_r}\left|\nabla_y\overline{u}_{\varepsilon}(z)-\i A\left(z\right)\overline{u}_{\varepsilon}(z)\right|_1 dz\right)\, \int_{0}^r \rho_m(s)s^{N-1} ds\\
\nonumber
&+\int_{B(0,r)}\int_0^1 \int_{\Omega}\left|\left( A\left(y+\frac{h}{2}\right)-A\left(y+th\right)\right)\cdot \frac{h}{|h|}\right|_1\left|\overline{u}_{\varepsilon}(y+th)\right|_1 \rho_m(h) dydtdh\\
\nonumber
&\leq Q_{1,N}\left(|D\overline{u}|_A(E_r')\int_{B(0,r)} \rho_m(h) dh +\varepsilon {\rm Lip}(A,E_r') \|\overline{u}\|_{L^1(E_r')}\right)\\
&+\frac{{\rm Lip}(A, E_r') \|\overline{u}\|_{L^1(W)}}{2}\int_{B(0,r)}|h|\rho_m(h)dh, \notag
\end{align}
where in the last inequality we used Lemma \ref{conv}. 
Putting together \eqref{stimaI} and \eqref{stimaII} we get
\begin{equation*}
\begin{aligned}
&\int_{\Omega}\int_{\Omega}\frac{|\overline{u}_{\varepsilon}(x)-e^{\i(x-y)A(\frac{x+y}{2})}\overline{u}_{\varepsilon}(y)|_1}{|x-y|}\rho_m(x-y) dxdy\\
&\leq Q_{1,N}\Big(|D\overline{u}|_A(E_r')\int_{0}^r \rho_m(s)s^{N-1} ds+\varepsilon {\rm Lip}(A,E_r') \|\overline{u}\|_{L^1(E_r')}\Big)
+\frac{{\rm Lip}(A,E_r') \|\overline{u}\|_{L^1(W)}}{2}\int_{0}^r s^N\rho_m(s)ds\\
&+C\left(\|\overline{u}\|_{BV_A(W')}+\eps{\rm Lip}(A,W')\|\overline{u}\|_{L^1(W')}\right)\left(\int_0^1 s^N\rho_m(s)ds+\int_{1}^{\infty} s^{N-1}\rho_m(s)ds\right) \\
&+\frac{C\|\overline{u}\|_{L^1(W)}}{r}\int_{r}^{\infty} s^{N-1} \rho_m(s) ds.
\end{aligned}\end{equation*}
The conclusion follows letting $\varepsilon\to 0^+$.
\end{proof}

\noindent
$\bullet$ {\bf Proof of Theorem \ref{Main} concluded.}
	Fix $r>0$, $W\Supset\Omega$ and let $\overline{u}=Eu\in BV_A(\R^N)$ 
	be an extension of $u$ such that $\overline{u}=0$ in $W^c$ and $|D\overline{u}|_A(\partial\Omega)=0$,
	according to Lemma~\ref{ExtDom-new}. Using Lemma \ref{approx} and 
	Lemma \ref{Ponciarello} for every $0<\varepsilon<r$ we have
\begin{align*}
&\int_{\Omega_r\cap B(0,1/r)}\int_{\Omega_r\cap B(0,1/r)}\frac{|u_{\varepsilon}(x)-e^{\i(x-y)\cdot A(\frac{x+y}{2})}u_{\varepsilon}(y)|_1}{|x-y|}\rho_m(x-y) dxdy\\
\nonumber
&\leq \int_{\Omega}\int_{\Omega}\frac{|u(x)-e^{\i(x-y)\cdot A(\frac{x+y}{2})}u(y)|_1}{|x-y|}\rho_m(x-y) dxdy\\
\nonumber
&+\frac{1}{\varepsilon^N}\int_{B(0,\varepsilon)}\eta\left(\frac{z}{\varepsilon}\right)\int_{\Omega}\int_{\Omega} \frac{\left|e^{\i(x-y)\cdot A(\frac{x+y}{2}+z)}u(y)-e^{\i(x-y)\cdot A(\frac{x+y}{2})}u(y)\right|_1}{|x-y|} \rho_m(x-y) dxdydz\\
\nonumber
& \leq Q_{1,N}|D\overline{u}|_A(E_r')\int_{0}^r \rho_m(s)s^{N-1}ds +\frac{{\rm Lip}(A,E_r') \|\overline{u}\|_{L^1(W)}}{2}\int_{0}^r s^N\rho_m(s)ds\\
&+C\|\overline{u}\|_{BV_A(W')}\Big(\int_0^1 s^N\rho_m(s)ds+\int_1^\infty s^{N-1}\rho_m(s)ds\Big) 
+\frac{C\|\overline{u}\|_{L^1(W)}}{r}\int_{r}^{\infty} s^{N-1} \rho_m(s) ds \notag \\
&+\frac{1}{\varepsilon^N}\int_{B(0,\varepsilon)}\eta\left(\frac{z}{\varepsilon}\right)\int_{\Omega}\int_{\Omega} \frac{\left|e^{\i(x-y)\cdot A(\frac{x+y}{2}+z)}u(y)-e^{\i(x-y)\cdot A(\frac{x+y}{2})}u(y)\right|_1}{|x-y|} \rho_m(x-y) dxdydz. \notag 
\end{align*}
Letting $m\to\infty$, using \eqref{Pociarellosmooth}, \eqref{lim-cond} and \eqref{prorho} we get
\begin{align*}
&Q_{1,N}|Du_{\varepsilon}|_A(\Omega_r\cap B(0,1/r))\leq \lim_{m\to\infty}\int_{\Omega}\int_{\Omega}\frac{|u(x)-e^{\i(x-y)\cdot A(\frac{x+y}{2})}u(y)|_1}{|x-y|}\rho_m(x-y) dxdy\\
\nonumber
&+\lim_{m\to \infty} \frac{1}{\varepsilon^N}\int_{B(0,\varepsilon)}\eta\left(\frac{z}{\varepsilon}\right)\int_{\Omega}\int_{\Omega} \frac{\left|e^{\i(x-y)\cdot A(\frac{x+y}{2}+z)}u(y)-e^{\i(x-y)\cdot A(\frac{x+y}{2})}u(y)\right|_1}{|x-y|} \rho_m(x-y) dxdydz\\
\nonumber
&\leq Q_{1,N}|D\overline{u}|_A(E_r')\\
&\quad +\lim_{m\to \infty} \frac{1}{\varepsilon^N}\int_{B(0,\varepsilon)}\eta\left(\frac{z}{\varepsilon}\right)\int_{\Omega}\int_{\Omega} \frac{\left|e^{\i(x-y)\cdot A(\frac{x+y}{2}+z)}u(y)-e^{\i(x-y)\cdot A(\frac{x+y}{2})}u(y)\right|_1}{|x-y|} \rho_m(x-y) dxdydz.
\end{align*}
Letting $\varepsilon\to 0^+$, using the
lower semi-continuity of the total variation and Lemma \ref{approx} we have
\begin{align*}
Q_{1,N}|Du|_A(\Omega_r\cap B(0,1/r))&\leq \lim_{m\to\infty}\int_{\Omega}\int_{\Omega}\frac{|u(x)-e^{\i(x-y)\cdot A(\frac{x+y}{2})}u(y)|_1}{|x-y|}\rho_m(x-y) dxdy  \\
&\leq Q_{1,N}|D\overline{u}|_A(E_r'),
\end{align*}
the assertion follows letting $r\searrow 0$ and observing that
\begin{align*}
\lim_{r\to 0^+}|Du|_A(\Omega_r\cap B(0,1/r))=\lim_{r\to 0^+}|D\overline{u}|_A(E_r')=|Du|_A(\Omega).
\end{align*}
Indeed, since $|Du|_{A}(\cdot)$ is a Radon measure,
then by inner regularity
\[
\lim_{r\to 0^+}|Du|_A(\Omega_r\cap B(0,1/r))=|Du|_A(\Omega),
\]
and by outer regularity
\[
\lim_{r\to 0^+}|D\overline{u}|_A(E_r')=|Du|_A(\overline{\Omega})=|Du|_A(\Omega).
\]
This concludes the proof.
\qed



\begin{thebibliography}{99}

\bibitem{Ambrosio} 
L. Ambrosio, 
{\it Metric space valued functions of bounded variation,} 
Ann. Sc. Norm. Sup. Pisa Cl. Sci. {\bf 17} (1990), 405--425. 

\bibitem{AmbDepMart}
L.\ Ambrosio, G.\ De Philippis, L.\ Martinazzi,
{\it $\Gamma-$convergence of nonlocal perimeter functionals,}
Manuscripta Math. {\bf 134} (2011), 377--403.

\bibitem{ABF} 
L.  Ambrosio, N. Fusco, D. Pallara, 
{\it Functions of bounded variation and free discontinuity problems,} The Clarendon Press, Oxford University Press, New York, 2000.

\bibitem{AG78}
G.\ Anzellotti, M.\ Giaquinta, {\it BV functions and traces}, 
Rend. Sem. Mat. Univ. Padova {\bf 60} (1978), 1--21.

\bibitem{arioliSz}
G.\ Arioli, A.\ Szulkin,
{\it A semilinear Schr\"odinger equation in the presence of a magnetic field}, Arch. Ration. Mech. Anal. {\bf 170} (2003), 277--295.

\bibitem{AHS}
J. Avron, I. Herbst, B. Simon, {\it Schr\"odinger operators with magnetic fields. I. General interactions}, Duke Math. J. {\bf 45} (1978), 847--883.

\bibitem{Barb} 
D. Barbieri, {\it Approximations of Sobolev norms in Carnot groups}, Comm.\ Contemp.\ Math.\ {\bf13} (2011), 765--794.

\bibitem{bourg}
J. Bourgain, H. Brezis, P. Mironescu, 
{\it Another look at Sobolev spaces},
in \emph{Optimal Control and Partial Differential Equations. A Volume in Honor of Professor Alain Bensoussan's 60th Birthday}
(eds. J. L. Menaldi, E. Rofman and A. Sulem), IOS Press, Amsterdam, 2001, 439--455.

\bibitem{bourg2}
J. Bourgain, H. Brezis, P. Mironescu,
{\it Limiting embedding theorems for $W^{s,p}$ when $s \uparrow 1$ and applications},
{J. Anal. Math.} \textbf{87} (2002), 77--101.

\bibitem{brapasq}
L.\ Brasco, E.\ Parini, M. Squassina,  
{\it Stability of variational eigenvalues for the fractional $p$-Laplacian},
Discrete Contin. Dyn. Syst. A {\bf 36} (2016), 1813--1845.

\bibitem{bre}
H. Brezis,
{\it How to recognize constant functions. Connections with Sobolev spaces},
Russian Mathematical Surveys \textbf{57} (2002), 693--708.

\bibitem{bre-linc}
H. Brezis, 
{\it New approximations of the total variation and filters in imaging}, Rend Accad. Lincei {\bf 26} (2015), 223--240.

\bibitem{BHN}
H. Brezis, Hoai-Minh Nguyen,
{\it Non-local functionals related to the total variation and connections with Image Processing}, preprint. 
\url{http://arxiv.org/abs/1608.08204}

\bibitem{BHN2}
H. Brezis, Hoai-Minh Nguyen,
{\it The BBM formula revisited}, Atti Accad. Naz. Lincei Rend.
Lincei Mat. Appl., to appear. 
\url{http://arxiv.org/abs/1606.05518}

\bibitem{BHN3}
H. Brezis, Hoai-Minh Nguyen,
{\it Two subtle convex nonlocal approximations of the BV-norm},
Nonlinear Anal. {\bf 137} (2016), 222–-245. 

\bibitem{CRS}
L.\ Caffarelli, J.\ -M.\ Roquejoffre, O.\ Savin,
{\it Nonlocal minimal surfaces},
Comm.\ Pure Appl. Math. {\bf 63} (2010), 1111--1144.

\bibitem{CV}
L.\ Caffarelli, E.\ Valdinoci, 
{\em Regularity properties of nonlocal 
minimal surfaces via limiting arguments}, Adv. Math. {\bf 248} (2013), 843--871. 


\bibitem{Cui}
X.\ Cui, N.\ Lam, G.\ Lu, {\it New characterizations of Sobolev spaces in the Heisenberg group},
J. Funct. Anal. \textbf{267} (2014), 2962--2994.

\bibitem{piemar}
P.\ d'Avenia, M.\ Squassina,  
{\it Ground states for fractional magnetic operators},  
ESAIM Control Optim. Calc. Var., to appear.

\bibitem{davila}
J.\ Davila, {\it On an open question about functions of bounded variation,} 
Calc. Var. Partial Differential Equations {\bf 15} (2002), 519--527. 

\bibitem{EG}
L.C.\ Evans and R.F.\ Gariepy, Measure theory and fine properties of functions,
Chapman and Hall/CRC; Revised ed. edition, 2015.


\bibitem{Dip}
S.\ Dipierro, A.\ Figalli, G.\ Palatucci, E.\ Valdinoci, 
{\it Asymptotics of the s-perimeter as $s \to 0$,} 
Discrete Contin. Dyn. Syst. {\bf 33} (2013), 2777--2790.

\bibitem{FerPin}
F.\ Ferrari, A.\ Pinamonti,
{\it Characterization by asymptotic mean formulas of $p$-harmonic functions in Carnot groups,}  Potential Anal. \textbf{42} (2015), 203--227.

\bibitem{I10}
T.\ Ichinose, {\it Magnetic relativistic Schr\"odinger operators and imaginary-time path integrals}, Mathematical physics, spectral theory and stochastic analysis, 247--297, Oper.\ Theory Adv.\ Appl. {\bf 232}, Birkh\"auser/Springer, Basel, 2013.

\bibitem{Spe}
G. Leoni, D. Spector, {\it Characterization of Sobolev and BV spaces,}
J. Funct. Anal. {\bf 261} (2011),  2926--2958.

\bibitem{Spe2}
G. Leoni, D. Spector, {\it Corrigendum to "Characterization of Sobolev and BV spaces'',}  J. Funct. Anal. {\bf 266} (2014), 1106--1114.

\bibitem{LL}
E.\ Lieb and M.\ Loss, Analysis, Graduate Studies in Mathematics {\bf 14}, 2001.

\bibitem{Ludwig1}
M.\ Ludwig, {\it Anisotropic fractional Sobolev norms}, 
Adv. Math. \textbf{252} (2014), 150--157.

\bibitem{Ludwig2}
M.\ Ludwig, {\it Anisotropic fractional perimeters},
J. Differential Geom. \textbf{96} (2014), 77--93. 

\bibitem{mazia}
V.\ Maz'ya and T.\ Shaposhnikova,
{\it On the Bourgain, Brezis, and Mironescu theorem concerning limiting embeddings of fractional Sobolev spaces},
J. Funct. Anal. \textbf{195} (2002), 230--238.

\bibitem{mazia2}
V.\ Maz'ya and T.\ Shaposhnikova,
{\it Erratum to: "On the Bourgain, Brezis and Mironescu theorem concerning limiting embeddings of fractional Sobolev spaces''}, 
J. Funct. Anal. \textbf{201} (2003), 298--300.  

\bibitem{nguyen06}
H.-M.\ Nguyen, {\it Some new characterizations of Sobolev spaces},
J. Funct. Anal. \textbf{237} (2006), 689--720.

\bibitem{Lim0}
A.\ Pinamonti, M.\ Squassina, E.\ Vecchi, {\it The Maz'ya-Shaposhnikova limit in the magnetic setting},
J. Math. Anal. Appl. {\bf 449} (2017), 1152--1159.

\bibitem{Ponce} 
A.\ Ponce, {\it A new approach to Sobolev spaces and connections to $\Gamma$-convergence,} 
Calc. Var. Partial Differential Equations {\bf 19} (2004), 229--255.

\bibitem{PonceSpector}
A.\ Ponce, D.\ Spector, {\it  On formulae decoupling the total variation of BV functions,}  Nonlinear Anal., to appear \url{http://dx.doi.org/10.1016/j.na.2016.08.028}

\bibitem{reed}
M.\ Reed, B.\ Simon, Methods of modern mathematical 
physics, I, Functional analysis, Academic Press, Inc.,
New York, 1980

\bibitem{BM}
M.\ Squassina, B.\ Volzone, 
{\it Bourgain-Brezis-Mironescu formula for magnetic operators,}
C. R. Math. Acad. Sci. Paris {\bf 354} (2016), 825--831.

\bibitem{Z}
W.\ P.\ Ziemer,
{\it Weakly Differentiable Functions. Sobolev Spaces and Functions of Bounded Variation}. Springer-Verlag New York, Inc. New York, NY, 1989.

\end{thebibliography}
\end{document}